\documentclass[11pt]{amsart}
\usepackage[margin=1in]{geometry}
\usepackage{amsmath,amssymb,amsthm,mathtools,booktabs,enumitem,array}
\usepackage{xcolor}
\usepackage[colorlinks,linkcolor=blue!50!black,citecolor=blue!50!black,urlcolor=blue!50!black]{hyperref}

\theoremstyle{plain}
\newtheorem{theorem}{Theorem}[section]
\newtheorem{proposition}[theorem]{Proposition}
\newtheorem{lemma}[theorem]{Lemma}
\newtheorem{corollary}[theorem]{Corollary}
\newtheorem{fact}[theorem]{Fact}
\newtheorem{question}[theorem]{Question}

\theoremstyle{definition}
\newtheorem{definition}[theorem]{Definition}
\newtheorem{remark}[theorem]{Remark}
\newtheorem{example}[theorem]{Example}

\newcommand{\F}{\mathbb{F}}
\newcommand{\Z}{\mathbb{Z}}
\newcommand{\Q}{\mathbb{Q}}
\newcommand{\PP}{\mathbb{P}}
\newcommand{\A}{\mathbb{A}}

\newcommand{\tw}[2]{{}^{#1}\!#2}
\newcommand{\gen}{\mathrm{gen}}
\newcommand{\tgen}{\tau_{\gen}}
\newcommand{\free}{\mathrm{free}}
\newcommand{\Gam}{\Gamma}
\newcommand{\Dec}{\mathrm{Dec}}
\newcommand{\bDec}{\overline{\mathrm{Dec}}}
\DeclareMathOperator{\ed}{ed}
\newcommand{\edd}[1]{\ed^{[#1]}}
\DeclareMathOperator{\rdim}{rdim}
\DeclareMathOperator{\trdeg}{trdeg}

\DeclareMathOperator{\im}{im}

\DeclareMathOperator{\Aut}{Aut}
\DeclareMathOperator{\Hom}{Hom}
\DeclareMathOperator{\Gal}{Gal}
\DeclareMathOperator{\Ind}{Ind}
\DeclareMathOperator{\Spec}{Spec}

\DeclareMathOperator{\PGL}{PGL}

\DeclareMathOperator{\Pic}{Pic}

\DeclareMathOperator{\id}{id}
\DeclareMathOperator{\rk}{rk}
\DeclareMathOperator{\Frac}{Frac}

\newcommand{\inn}{\operatorname{in}}
\newcommand{\Fix}{\mathrm{Fix}}
\newcommand{\kb}{\bar k}
\newcommand{\Om}{\Omega}
\usepackage{longtable}
\theoremstyle{plain}
\newtheorem*{thmA}{Theorem A}
\newtheorem*{thmB}{Theorem B}
\newtheorem*{thmC}{Theorem C}

\title[The finite-degree profile of essential dimension]{The finite-degree profile of essential dimension I: general theory and characteristic prime to $|G|$}
\author{Abhishek Shukla}
\address{Department of Statistics \& Actuarial Science, Simon Fraser University, Burnaby, BC, Canada}
\email{shukla@sfu.ca}
\date{\today}

\begin{document}

\begin{abstract}
Essential dimension counts the parameters needed to define a torsor, but it can drop sharply over a field extension, and it does not say how large an extension is needed. For a $G$-torsor $\tau$ over a field $F\supseteq k$ we study the \emph{finite-degree profile} $\edd d_k(\tau)=\min\{\ed_k(\tau\otimes_FF'):[F':F]\le d\}$, a variant of invariants of Farb--Kisin--Wolfson and Farb--Wolfson, and its jump degrees $d_j(\tau)$, the least degree of an extension over which $\tau$ needs only $j$ parameters. The profile is geometric over every field: $\edd d_k(\tau)\le j$ iff some generically free $G$-variety of dimension $\le j$ has a closed point of degree $\le d$ in the free locus of its twist by $\tau$. We prove a Sylow divisibility theorem for jump degrees and a valuation-theoretic rank--index theorem, $\edd d_k(\tgen)\ge\min\{\rk B:B\le A,\ [A:B]\le d\}$ for abelian $A$ of order prime to $\operatorname{char}k$, which determines the profile of every abelian group over a field with enough roots of unity, for instance $n-\lfloor\log_2d\rfloor$ for $(\Z/2)^n$. For $\Z/p$ in characteristic $0$ we show $\lceil m/2\rceil\le d_1\le[k(\zeta_p+\zeta_p^{-1}):k]$, $m=[k(\zeta_p):k]$, with equality for $m$ even. For $S_5$ in characteristic $0$ we prove $d_1=4$, so Klein's reduction of the generic quintic is optimal; the lower bound comes from Abel--Jacobi rigidity, not from cohomological invariants. For $n\ge6$, $6\mid d_1(S_n)$.
\end{abstract}
\maketitle
\tableofcontents

% ---- begin intro.tex ----
\section{Introduction}\label{sec:intro}

\subsection*{Essential dimension and base change}
Let $k$ be a field, $G$ a finite group and $\tau$ a $G$-torsor over a field $F\supseteq k$. The \emph{essential dimension} $\ed_k(\tau)$ is the least transcendence degree over $k$ of a subfield of $F$ over which $\tau$ is defined, and $\ed_k(G)=\sup_\tau\ed_k(\tau)$ is attained by the generic torsor $\tgen$ \cite{BR97,BF03,Mer13}. The invariant is very sensitive to the base field. It never increases under base change, and after suitable extensions it collapses: over a field containing the roots of unity of order the exponent, a torsor for an abelian group descends to transcendence degree $\le\rk G$, over a $q$-closed field the essential dimension is the essential dimension at $q$, and in characteristic $p$ every $p$-group torsor over a large field has $\ed_k\le1$ \cite{RS22,RV18}. Ledet's conjecture $\ed_k(\Z/p^n)=n$ for $k$ algebraically closed of characteristic $p$ \cite{Ledet04} is the sharpest instance of this instability: the invariant would grow with $n$ over the base field and be $\le1$ over every large extension. Such statements say that the invariant drops after some extension. They do not say how large an extension is needed. This paper studies that question degree by degree.

\subsection*{The profile}
For $d\ge1$ put
\[
\edd d_k(\tau):=\min\{\ed_k(\tau\otimes_FF')\;:\;F'/F\text{ a field extension},\ [F':F]\le d\}.
\]
The function $d\mapsto\edd d_k(\tau)$ is a non-increasing step function, the \emph{finite-degree profile} of $\tau$ (Definition~\ref{def:profile}). It starts at $\ed_k(\tau)$ and reaches $0$ at the degree of a splitting field (Proposition~\ref{prop:d0}). Its steps are the \emph{jump degrees}
\[
d_j(\tau)=\min\{d:\edd d_k(\tau)\le j\},
\]
the least degree of an extension over which $\tau$ needs only $j$ parameters. For a prime $q$ the \emph{prime-to-$q$ profile} restricts to extensions of degree prime to $q$. The profile of $\tgen$ does not depend on the choice of faithful representation and dominates the profile of every torsor over every field containing $k$ (Proposition~\ref{prop:vers}); so $\edd d_k(\tgen)$ is an invariant of $(G,k)$ refining $\ed_k(G)$. Examples: for $\Z/2$ it is $(1,0)$; for $(\Z/2)^3$ over a field of characteristic $\ne2$ it is $(3,2,2,1,1,1,1,0)$ (Corollary~\ref{cor:A8}); for the symmetric groups $S_n$, $n\le5$, it is listed in Table~\ref{tab:A}.

\subsection*{Results}
Geometry and Sylow divisibility. The profile has a geometric meaning: $\edd d_k(\tau)\le j$ iff some generically free $G$-variety of dimension $\le j$ has a closed point of degree $\le d$ in the free locus of its twist by $\tau$ (Theorem~\ref{thm:geom}). Sylow divisibility (Theorem~\ref{thm:div}) says that, for $k$ perfect, every prime $q$ with $\ed_{k,q}(G)>j$ divides $d_j(\tgen)$. It is blind to $q$-power degrees. The following results are lower bounds that do see them.

\begin{thmA}[Theorem~\ref{thm:rankindex}]
Let $k$ be any field, $G$ a finite group and $A\le G$ abelian with $\operatorname{char}k\nmid|A|$. For $\rho_A(d):=\min\{\rk B:B\le A,\ [A:B]\le d\}$ we have $\ed_k(\tgen\otimes_KF')\ge\rho_A([F':K]_s)$ for every finite extension $F'/K$; in particular $\edd d_k(\tgen)\ge\rho_A(d)$.
\end{thmA}
The proof is valuation-theoretic: for a flag valuation on $L=k(V)$ whose decomposition group contains $A$, the decomposition group over $F'$ meets $A$ in a subgroup of index $\le[F':K]_s$, while over a field of transcendence degree $j$ with algebraically closed residue field the abelian subgroups of order prime to the residue characteristic in the image of a decomposition group have rank $\le j$. Together with Kummer theory this determines the profile of every abelian group $G$ of exponent $e$ over a field of characteristic prime to $e$ containing $\mu_e$, namely $\edd d_k(\tgen)=\min\{\rk B:[G:B]\le d\}$ (Corollary~\ref{cor:A7}), and that of $(\Z/q)^n$ whenever $\ed_k(\Z/q)\le1$, in particular of $(\Z/2)^n$ over every field of characteristic $\ne2$: $\edd d_k(\tgen)=n-\lfloor\log_2d\rfloor$ (Corollary~\ref{cor:A8}). For cyclic groups Theorem~A gives nothing beyond the last jump degree (Remark~\ref{rem:limits}).

\begin{thmB}[Theorem~\ref{thm:Zp}]
Let $\operatorname{char}k=0$, $p$ an odd prime, $\zeta$ a primitive $p$-th root of unity, $c=\zeta+\zeta^{-1}$ and $m=[k(\zeta):k]$. Then $\lceil m/2\rceil\le d_{1,k}(\tgen(\Z/p))\le[k(c):k]$; if $m$ is even the bounds meet, $d_1=m/2$. In particular $d_{1,\Q}(\tgen(\Z/p))=(p-1)/2$.
\end{thmB}
The upper bound holds for every $\Z/p$-torsor over every field containing $k$ (Corollary~\ref{cor:pertorsor}); the lower bound combines a fixed point produced by the valuation argument of Theorem~A with Riemann--Hurwitz and Castelnuovo's inequality. The case $m\ge3$ odd remains open (Question~\ref{q:sharpodd}).

\begin{thmC}[Theorem~\ref{thm:S5}]
Let $\operatorname{char}k=0$. Then $\edd d_k(\tgen(S_5))=2$ for $d\le3$, $=1$ for $4\le d\le119$ and $=0$ for $d\ge120$. In particular $d_1(\tgen(S_5))=4$.
\end{thmC}
Thus the generic quintic needs $2$ parameters over every extension of degree $\le3$ and $1$ parameter over a suitable extension of degree $4$: Klein's reduction of the quintic through $K(\sqrt\Delta)$ and a conic \cite{Kl84} (see \cite[Thm.~4.2.2]{FKW23} and \cite[Thm.~1.3]{FW25}, which require $\sqrt5\in k$) is optimal. Over fields without $\sqrt5$ we give a variant through $K(\sqrt{5\Delta})$ and a twisted conic, using a semilinear action of $S_5$ on a conic over $k(\sqrt5)$ viewed as a $k$-curve. Theorem~A cannot give the lower bound (for $S_5$ it only recovers $\ed(S_5)\ge2$ at $d=1$). The lower bound is obtained from Abel--Jacobi rigidity (Section~\ref{sec:const}) rather than from cohomological invariants (Remark~\ref{rem:S5coh}): over $L=k(V)$ the twist becomes constant, a point of degree $d$ on a twisted curve becomes a divisor class on a curve over $k$, and since $L$ is purely transcendental over $k$ that class is constant, hence $G$-invariant. For $n\ge6$, Sylow divisibility gives $6\mid d_1(\tgen(S_n))$ (Corollary~\ref{cor:SnSylow}); so adjoining a root of an equation of degree $\le5$ never reduces the generic polynomial of degree $n$ to one parameter. For $S_6$ only partial information is known (Example~\ref{ex:Sn}).

\begin{table}[tbp]
\centering\small
\renewcommand{\arraystretch}{1.25}
\begin{tabular}{@{}>{\raggedright\arraybackslash}p{0.13\textwidth}>{\raggedright\arraybackslash}p{0.27\textwidth}>{\raggedright\arraybackslash}p{0.37\textwidth}>{\raggedright\arraybackslash}p{0.16\textwidth}@{}}
\toprule
$G$ & $k$ & profile of $\tgen$ & reference\\
\midrule
abelian, exponent $e$ & $\operatorname{char}k\nmid e$, $\mu_e\subset k$ & $\edd d_k=\min\{\rk B:B\le G,\ [G:B]\le d\}$ & Cor.~\ref{cor:A7}\\
$(\Z/q)^n$, $q$ prime & $\operatorname{char}k\ne q$, $\ed_k(\Z/q)\le1$ & $\edd d_k=n-\lfloor\log_qd\rfloor$ for $d\le q^n$; $d_j=q^{n-j}$ & Cor.~\ref{cor:A8}(b)\\
$(\Z/2)^n$ & any $k$, $\operatorname{char}k\ne2$ & $\edd d_k=n-\lfloor\log_2d\rfloor$ for $d\le2^n$ & Cor.~\ref{cor:A8}(c)\\
$p$-group & $\operatorname{char}k\ne p$, $\zeta_p\in k$, $k$ perfect & $\edd d_k=\ed_k(G)$ for $d\le p-1$; drops only at multiples of $p$ & Prop.~\ref{prop:KMplateau}(c)\\
$\Z/p$, $p$ odd & $\operatorname{char}k=0$, $\zeta_p\in k$ & $d_1=1$ & Cor.~\ref{cor:sharp}(i)\\
$\Z/p$, $p$ odd & $\operatorname{char}k=0$, $m$ even & $d_1=m/2=[k(c):k]$ & Cor.~\ref{cor:sharp}(ii)\\
$\Z/p$, $p$ odd & $\operatorname{char}k=0$, $m\ge3$ odd & $\frac{m+1}2\le d_1\le m$ & Cor.~\ref{cor:sharp}(iii)\\
$\Z/p$, $p$ odd & $k=\Q$ & $d_1=(p-1)/2$, $d_0=p$ & Thm.~\ref{thm:Zp}\\
$A_5$ & $\operatorname{char}k=0$, $\sqrt5\in k$ & $\edd1=2$, $\edd d=1$ for $2\le d\le59$, $\edd{60}=0$; $d_1=2$ & Cor.~\ref{cor:A5}\\
$S_2=\Z/2$ & $\operatorname{char}k=0$ & $(1,0)$ & Prop.~\ref{prop:Sn4}\\
$S_3$ & $\operatorname{char}k=0$ & $\edd d=1$ for $d\le5$, $\edd6=0$ & Prop.~\ref{prop:Sn4}\\
$S_4$ & $\operatorname{char}k=0$ & $\edd1=2$, $\edd d=1$ for $2\le d\le23$, $\edd{24}=0$; $d_1=2$ & Prop.~\ref{prop:Sn4}\\
$S_5$ & $\operatorname{char}k=0$ & $\edd d=2$ for $d\le3$, $1$ for $4\le d\le119$, $\edd{120}=0$; $d_1=4$ & Thm.~\ref{thm:S5}\\
\bottomrule
\end{tabular}
\caption{The finite-degree profile of the generic torsor in characteristic prime to $|G|$ ($q$ and $p$ primes; $d_j$ taken relative to $k$). }\label{tab:A}
\end{table}

\subsection*{Relation to earlier work}
Measuring essential dimension after passing to extensions of bounded degree is not a new idea, and the profile defined here is a variant of an existing invariant; we do not claim it as new. Farb--Kisin--Wolfson \cite{FKW23} study essential dimension relative to a class of extensions. Farb--Wolfson \cite{FW25}, over a perfect field $k$, define for a faithful $G$-variety $X$ the quantity $\ed_k(X;\le n)$ as the least $\dim Y$ over faithful $G$-varieties $\tilde X$, $Y$ for which there are a branched cover $\tilde X\dashrightarrow X$ of degree $\le n$ (a generically finite dominant rational map) and a $G$-compression $\tilde X\dashrightarrow Y$, and take the supremum over $X$ to obtain $\ed_k(G;\le n)$ (the reduction to linear $X$ is \cite[footnote~4]{FW25}, citing \cite[Ex.~4.6, Lemma~4.9]{FKW23}). In torsor language such a cover is an extension $F'/K$ of degree $\le n$ for which $L\otimes_KF'$ is a field, an irreducible cover; our $\edd d_k(\tgen)$ also allows extensions for which $L\otimes_KF'$ is a product of fields. Hence $\edd d_k(\tgen)\le\ed_k(G;\le d)$, with equality when the optimum is attained by a connected cover. This dictionary is not used in any proof below. Klein's witness for $S_5$ is of the reducible kind, so Theorem~C does not determine $\ed_k(S_5;\le4)$ in the sense of \cite{FW25}; what the proof of Proposition~\ref{prop:S5lower} gives, in the case $H=S_5$, is $\ed_k(S_5;\le n)\ge2$ for $n\le3$.

The geometric characterisation (Theorem~\ref{thm:geom}), the versality of the generic torsor (Proposition~\ref{prop:vers}), Sylow divisibility (Theorem~\ref{thm:div}) and the Karpenko--Merkurjev plateau (Proposition~\ref{prop:KMplateau}) are formal consequences of results of Duncan--Reichstein \cite{DR15}, Reichstein--Scavia \cite{RS22} and Karpenko--Merkurjev \cite{KM}; Sylow divisibility is close to the divisibility bound of \cite[Rem.~1.9]{FW25}. We record them for use below. The contributions of this paper are Theorems~A, B and C and the explicit profiles computed from them. The case $d=1$ of Theorem~A, $\ed\ge\rk A$, is classical \cite{BR97,BF03,RY}; the control of the index is new, and we do not find an analogue of it in \cite{FW25}. Theorem~B treats $\Z/p$, which \cite{FW25} does not treat; it uses the same Castelnuovo-type gonality argument as \cite{FW25} (their Lemma~2.2 and Corollary~2.3), together with Riemann--Hurwitz and a valuation-produced fixed point. For Theorem~C, \cite[Thm.~1.6]{FW25} requires, among other things, that $G$ have no proper subgroup of index $\le n$ and that $G$ not act nontrivially on a smooth curve of small genus; $S_5$ fails the first condition ($A_5$ has index $2$) and $A_5$ fails the third (it acts on $\PP^1$), so the value $d_1(S_5)=4$ is not in \cite{FW25}.

The profile is also close in spirit to the resolvent degree of a finite group \cite{Br75,FW19,BRT} and to the classical reductions of Klein, Hermite and Bring--Jerrard of the generic quintic. The profile fixes the torsor $\tau$ and bounds the degree of a single extension of its base field, measuring the result by essential dimension; resolvent degree, roughly, allows towers of extensions. The jump degree $d_1$ is an invariant of a single step, and we do not compare it with the resolvent degree of the quintic.

\subsection*{Why a profile}
The profile does not compute $\ed_k(G)$: by monotonicity its slots with $d>1$ only bound $\edd1$ from below. It offers a graded way to refine a global question. ``Is $\ed_k(\tgen)\le j$?'' is one question; ``is there an extension of degree $\le d$ over which $\tgen$ needs $\le j$ parameters?'' is a family of questions, each geometric by Theorem~\ref{thm:geom} and bounded: for $j=1$ the varieties are curves and a point of degree $d$ is a divisor of degree $d$, to which Riemann--Roch, Clifford's inequality and the structure of Jacobians apply. It also gives information that Sylow divisibility cannot see, namely the $q$-power part of the profile: Theorems~A, B and C are lower bounds exactly there.

\subsection*{Roadmap}
Section~\ref{sec:conv} fixes conventions. Sections~\ref{sec:profile}--\ref{sec:sylow} develop the formal properties of the profile, its geometric meaning and Sylow divisibility, and they hold in every characteristic. Sections~\ref{sec:const}--\ref{sec:S5} prove Theorems~A, B and C: Section~\ref{sec:const} collects the rigidity tools, Section~\ref{sec:thmA} proves the rank--index theorem and its consequences for abelian groups, Section~\ref{sec:thmB} proves Theorem~B, and Section~\ref{sec:S5} proves Theorem~C together with the Sylow bound and the geometric lower bound for $S_n$, $n\ge5$, the profile of $A_5$, and the profiles of $S_n$, $n\le4$. Section~\ref{sec:open} lists open problems.

% ---- end intro.tex ----

% ---- begin conv.tex ----
\section{Conventions and background}\label{sec:conv}

\subsection{Essential dimension}\label{ss:ed}
Throughout, $k$ is a field, $G$ a finite group (constant group scheme over $k$), $F\supseteq k$ a field and $\tau\colon T\to\Spec F$ a $G$-torsor. Some statements below need $k$ perfect or algebraically closed, and say so.

The classifying homomorphism $\varphi_\tau\colon\Gam_F\to G$ (Section~\ref{ss:galois}) is defined up to conjugation, its image $H=\im\varphi_\tau$ up to conjugacy, and $T\cong\coprod_{G/H}\Spec E$ with $E/F$ Galois with group $H$ (the \emph{splitting field} of $\tau$). Following \cite{BR97,BF03,Mer13}, $\ed_k(\tau)$ is the least $\trdeg_kF_0$ over subfields $k\subseteq F_0\subseteq F$ such that $\tau\cong\tau_0\otimes_{F_0}F$ for a $G$-torsor $\tau_0$ over $F_0$, and $\ed_k(G)=\sup_\tau\ed_k(\tau)$. We use, in every characteristic:
\begin{itemize}[leftmargin=3em]
\item[(B1)] $\ed_k(\tau\otimes_FF')\le\ed_k(\tau)$ for every field extension $F'/F$ \cite[Lemma~2.1(a)]{RS22};
\item[(B2)] $\ed_k(\tau)\ge\ed_{k'}(\tau)$ for $k\subseteq k'\subseteq F$ \cite[Lemma~2.1(b)]{RS22}, with equality if $k'/k$ is algebraic: a descent of $\tau$ to $F_0\supseteq k$ gives a descent to $F_0k'\subseteq F$, and $\trdeg_{k'}F_0k'=\trdeg_kF_0$;
\item[(B3)] $\ed_k(\tau)=0$ iff $\tau$ descends to the algebraic closure of $k$ in $F$; for $k$ algebraically closed, iff $\tau$ is split. For general $k$ this is not ``descends to $k$'': $k=\Q$, $F=\Q(i)(t)$, $\tau=F(\sqrt{1+i})$ has $\ed_\Q=0$ but is not $F(\sqrt a)$ with $a\in\Q$, since $N_{\Q(i)/\Q}(1+i)=2$ is not a square;
\item[(B4)] purely transcendental base change over an infinite field does not change $\ed_k$ (Lemma~\ref{lem:constant}(i)).
\end{itemize}
For a prime $q$, the \emph{essential dimension at $q$} is
\[
\ed_{k,q}(\tau)=\min\{\ed_k(\tau\otimes F'):F'/F\text{ finite separable},\ q\nmid[F':F]\},\qquad\ed_{k,q}(G)=\sup_\tau\ed_{k,q}(\tau);
\]
equivalently $\ed_{k,q}(\tau)=\ed_k(\tau\otimes F^{(q)})$ for a $q$-closure $F^{(q)}$ of $F$ \cite[(2.5)]{RS22} (Lemma~\ref{lem:qcl}).

Let $Y$ be a generically free primitive $G$-variety over $k$ (primitive: $G$ permutes the irreducible components transitively). A $G$-compression is a dominant $G$-equivariant rational map $Y\dashrightarrow X$ to a generically free primitive $G$-variety, and $\ed_k(Y)$ is the minimal $\dim X$ over compressions (for finite $G$). Shrinking $Y$ to a $G$-torsor $Y^0\to B$ with $k(B)=k(Y)^G=:F$ and pulling back to the generic point gives a $G$-torsor $\tau_Y$ over $F$, and
\begin{equation}\label{eq:23}
\ed_k(Y)=\ed_k(\tau_Y)\qquad\text{\cite[(2.3)]{RS22}, \cite[Lemma~3.9]{Mer13}.}
\end{equation}
The \emph{generic torsor} $\tgen\colon\Spec L\to\Spec K$, $L=k(V)$, $K=L^G$, is taken for a faithful representation $V$ of $G$ over $k$; $\ed_k(\tgen)=\ed_k(G)$ \cite{BF03,Mer13}, and the profile of $\tgen$ does not depend on the choice of faithful representation $V$ (Proposition~\ref{prop:vers}). A correspondence $X\rightsquigarrow Z$ of degree $d$ between $G$-varieties is a diagram $X\xleftarrow{\ \pi\ }X'\dashrightarrow Z$ of $G$-equivariant rational maps with $\pi$ dominant of degree $d$ \cite[(2.6)]{RS22}.

\subsection{Galois groups, torsors and valuations}\label{ss:galois}
Fields and closures. All fields occurring in one argument are placed inside one algebraically closed field $\Om$, and for a subfield $M\subseteq\Om$ we put $\Gam_M:=\Aut(\Om/M)$, so that $M\subseteq M'$ gives $\Gam_{M'}\subseteq\Gam_M$. If $\Om$ is algebraic over $M$, restriction gives $\Gam_M\cong\Gal(M^s/M)$, where $M^s$ is the separable closure of $M$ in $\Om$, because $\Om/M^s$ is purely inseparable. In general (for instance for a subfield $E$ of a finitely generated field $F'\subseteq\Om$, where $\Om$ is not algebraic over $E$) restriction is only a surjection $\rho_E\colon\Gam_E\to\Gal(E^s/E)$; the classifying homomorphism of a torsor over $E$ factors through it.

Torsors \cite[I~\S5]{Se}. A $G$-torsor over a field $M$ is $\Spec P$ for a $G$-Galois $M$-algebra $P$. Choosing an $M$-algebra map $x\colon P\to\Om$ gives a homomorphism $c_x\colon\Gam_M\to G$, defined by $\gamma\circ x=x\circ c_x(\gamma)$. Changing $x$ conjugates $c_x$. Moreover $c_x=\bar c_x\circ\rho_M$ for a continuous homomorphism $\bar c_x\colon\Gal(M^s/M)\to G$, because $x$ takes values in $M^s$, and isomorphism classes of $G$-torsors over $M$ correspond bijectively to continuous homomorphisms $\Gal(M^s/M)\to G$ modulo $G$-conjugacy (via $\bar c_x$); the classifying homomorphism $\varphi_\tau$ of Section~\ref{ss:ed} is $c_x$ for any choice of $x$. If $M\subseteq M'\subseteq\Om$, the torsor $P\otimes_MM'$ corresponds to $c_x|_{\Gam_{M'}}$. Below only images of subgroups up to conjugacy matter, and these do not depend on left/right conventions.

The generic torsor as a homomorphism. Let $V$ be a faithful finite-dimensional representation of $G$ over $k$, $L=k(V)$ with the induced action $(gf)(v)=f(g^{-1}v)$, and $K=L^G$, so that $\tgen=(L/K)$ is a $G$-torsor by Artin's theorem. Embed $L$ in an algebraic closure $\Om$. The inclusion $L\hookrightarrow\Om$ gives the surjective homomorphism
\[
\varphi:=c_{\mathrm{incl}}\colon\Gam_K\to G,\qquad \varphi(\gamma)=\gamma|_L\in\Aut(L/K)=G.
\]

Rank. For a finite abelian group $B$, $\rk B$ is the minimal number of generators of $B$. It equals $\max_\ell\dim_{\F_\ell}B[\ell]$ and $\max_\ell\dim_{\F_\ell}B/\ell B$ ($\ell$ prime), and subgroups and quotients of $B$ have rank $\le\rk B$.

Valuations. Valuations are Krull valuations, written additively. For a valuation $\nu$ on $\Om$ and a subfield $M\subseteq\Om$ we write $\nu|_M$ for the restriction, $\kappa(M)$ for its residue field, $\nu(M^\times)$ for its value group, and
\[
\Dec_M:=\{\gamma\in\Gam_M:\ \nu\circ\gamma=\nu\}
\]
for the decomposition group. For $M\subseteq M'\subseteq\Om$ one has $\Dec_{M'}=\Dec_M\cap\Gam_{M'}$. Decomposition groups of all the fields occurring in an argument are taken with respect to restrictions of the same valuation $\nu$ on $\Om$; this identity of decomposition groups is what allows them to be compared (Proposition~\ref{prop:A4}) without choosing extensions of valuations to intermediate fields and conjugating. The order in which $\nu$ and the fields are chosen plays no role.

\subsection{Twists, free locus, versality, large fields}\label{ss:twist}
For a $G$-variety $X$ over $k$ and a $G$-torsor $\tau\colon T\to\Spec F$, the twist is $\tw{\tau}{X}=(T\times_kX)/G$, an algebraic space over $F$ (a variety if $X$ is quasi-projective, \cite{DR15}; for algebraic spaces see \cite[\S7]{RS22}). It commutes with base change in $F$, and for a $G$-stable open $U\subseteq X$, $\tw{\tau}{U}\subseteq\tw{\tau}{X}$ is an open subspace \cite[Cor.~3.4]{DR15}. We quote from \cite{RS22}:

\begin{fact}[{\cite[\S7, Lemmas 7.3, 7.4]{RS22}}]\label{fact:rs}
Let $X$ be a $G$-variety, $\tau\colon T\to\Spec F$ a $G$-torsor. $X$ is \emph{$\tau$-weakly versal} if there is a $G$-equivariant morphism $T\to X_F$ over $F$, and \emph{$\tau$-versal} if every dense $G$-invariant open subvariety is $\tau$-weakly versal.
\begin{enumerate}[label=\textup{(\alph*)}]
\item $X$ is $\tau$-weakly versal iff $\tw{\tau}{X}(F)\ne\emptyset$. More precisely $\tw{\tau}{X}(F)=\Hom_G(T,X_F)=\Hom_G^k(T,X)$.
\item If $X$ is smooth and $F$ is large, $X$ is weakly $\tau$-versal iff $X$ is $\tau$-versal.
\item If $X$ is generically free and primitive and $\tau$-versal, then $\ed_k(X)\ge\ed_k(\tau)$.
\end{enumerate}
A field $F$ is \emph{large} if every irreducible $F$-variety with a smooth $F$-point has dense $F$-points; $F((t))$, henselian nontrivially valued fields, and $q$-closed fields are large \cite[\S7]{RS22}, \cite{Pop96,Pop14}.
\end{fact}

For a $G$-variety $X$ put $X^{\free}:=X\smallsetminus\bigcup_{1\ne h\in G}X^h$, a $G$-stable open subset. The \emph{free locus of the twist} is $\tw{\tau}{(X^{\free})}\subseteq\tw{\tau}{X}$; an $F$-point of $\tw{\tau}{X}$ is free if it lies there, i.e.\ (by Fact~\ref{fact:rs}(a) applied to $X^{\free}$) if the corresponding equivariant morphism $T\to X$ factors through $X^{\free}$. A closed point $x$ of $\tw{\tau}{X}$ has \emph{degree} $[\kappa(x):F]$.

\begin{lemma}\label{lem:freepoint}
Let $Z$ be a quasi-projective generically free primitive $G$-variety over $k$ and $\tau$ a $G$-torsor over $F$. If $\tw{\tau}{Z}$ has a free $F$-point, then $\ed_k(\tau)\le\dim Z$.
\end{lemma}
\begin{proof}
$G$ acts freely on the quasi-projective variety $Z^{\free}$, so the quotient $\pi\colon Z^{\free}\to B:=Z^{\free}/G$ exists and is a $G$-torsor \cite[Exp.~V, \S\S1--2]{SGA1}.
A free $F$-point is a $G$-equivariant $f\colon T\to Z^{\free}$; it induces $\bar f\colon\Spec F\to B$ with $f^*(\pi)\cong\tau$ (a $G$-equivariant map of torsors over $\Spec F$ is an isomorphism). If $b$ is the image point of $\bar f$, then $\tau$ is the base change of the torsor $\pi^{-1}(b)\to\Spec\kappa(b)$ along $\kappa(b)\subseteq F$, and $\trdeg_k\kappa(b)\le\dim B=\dim Z$.
\end{proof}

\subsection{Torsors: decomposition and induction}\label{ss:tors}
A $G$-torsor over a field $F$ is a $G$-Galois $F$-algebra $P$. We use three standard facts.
\begin{enumerate}[label=(T\arabic*)]
\item $P=\prod_{i\in I}P^{(i)}$ is a product of fields permuted transitively by $G$; if $H_0$ is the stabiliser of $i_0$, then $P^{(i_0)}/F$ is Galois with group $H_0$ and $P\cong\operatorname{Ind}_{H_0}^GP^{(i_0)}$.
\item A $G$-equivariant $F$-algebra homomorphism between $G$-torsors over $F$ is an isomorphism.
\item Induction commutes with base change and with descent: if $T\cong T_0\otimes_EF$ then $\operatorname{Ind}_H^GT\cong\operatorname{Ind}_H^GT_0\otimes_EF$; so $\ed_k(\operatorname{Ind}_H^GT)\le\ed_k(T)$.
\end{enumerate}
For a finite extension $F'/K$ with $L/K$ a $G$-Galois extension (e.g.\ the generic torsor) write $L\otimes_KF'=\prod_{G/H}N_0$ as in (T1): $N_0$ is a field containing $L$ and $F'$, $N_0/F'$ is Galois with group $H$, and
\[[N_0:L]=[F':K]/[G:H].\]
In particular $[G:H]$ divides $[F':K]$.

% ---- end conv.tex ----

% ---- begin profile.tex ----
\section{The profile and its formal properties}\label{sec:profile}

\begin{definition}\label{def:profile}
Let $\tau$ be a $G$-torsor over a field $F\supseteq k$ and $d\ge1$. Put
\begin{gather*}
\edd d_k(\tau):=\min\{\ed_k(\tau\otimes_FF')\;:\;F'/F\text{ a field extension},\ [F':F]\le d\},\\
d_j(\tau):=\min\{d\ge1:\edd d_k(\tau)\le j\}\qquad(0\le j\le\ed_k(\tau)).
\end{gather*}
The sequence $(\edd d_k(\tau))_{d\ge1}$ is the \emph{finite-degree profile} of $\tau$ and the $d_j(\tau)$ are its \emph{jump degrees}; they depend on $k$, and when the base field varies we write $d_{j,k}(\tau)$. For a prime $q$ we also write $\edd d_k{}'(\tau):=\min\{\ed_k(\tau\otimes F'):[F':F]\le d,\ q\nmid[F':F]\}$, the \emph{prime-to-$q$ profile}, so that $\edd d_k\le\edd d_k{}'$, and $d_j^{(q')}(\tau):=\min\{d:\edd d_k{}'(\tau)\le j\}$ ($=\infty$ if no such $d$ exists), the \emph{prime-to-$q$ jump degrees}: $d_j^{(q')}(\tau)$ is the least degree prime to $q$ of an extension over which $\tau$ needs at most $j$ parameters. (The minimum of $d$ prime to $q$ with $\edd d_k(\tau)\le j$ would be a different and uninformative number, since $\edd d_k$ minimises over all extensions of degree $\le d$, including those of degree divisible by $q$.) 
\end{definition}

Only field extensions $F'/F$ enter the definition; étale algebras appear below only as total spaces of torsors.

\begin{proposition}[shape]\label{prop:shape}
Let $\tau$ be a $G$-torsor over $F\supseteq k$ and $H=\im\varphi_\tau$.
\begin{enumerate}[label=\textup{(\alph*)}]
\item $d\mapsto\edd d_k(\tau)$ is non-increasing with values in $\{0,\dots,\ed_k(\tau)\}$, $\edd1_k(\tau)=\ed_k(\tau)$, and $d_{\ed_k(\tau)}=1\le d_{\ed_k(\tau)-1}\le\dots\le d_0\le|H|$.
\item If $F'/F$ is a field extension of degree $e$ and $d\ge1$, then $\edd{de}_k(\tau)\le\edd d_k(\tau\otimes F')\le\edd d_k(\tau)$; hence $d_j(\tau)\le[F':F]\,d_j(\tau\otimes F')$ for every finite $F'/F$ and every $j$.
\item $\edd d_k(\tau)=\min\{j:d_j(\tau)\le d\}$: the jump degrees determine the profile.
\item If $k\subseteq k'\subseteq F$ then $\edd d_{k'}(\tau)\le\edd d_k(\tau)$ for all $d$, with equality if $k'/k$ is algebraic.
\item $\edd d_k(\tau)=0$ iff some $F'/F$ of degree $\le d$ makes $\tau\otimes F'$ descend to the algebraic closure of $k$ in $F'$; for $k$ algebraically closed, iff $\tau$ is split by an extension of degree $\le d$.
\end{enumerate}
\end{proposition}
\begin{proof}
(a) Over the splitting field $E$, $[E:F]=|H|$, the torsor is split, so $d_0\le|H|$; the rest is clear from the definition and (B1). (b) The first inequality: an extension $F''/F'$ of degree $\le d$ has $[F'':F]\le de$. The second: if $[F_1:F]\le d$ then the compositum $F_1F'$ (any composite field) has $[F_1F':F']\le d$ and $\ed_k(\tau\otimes F_1F')\le\ed_k(\tau\otimes F_1)$. If $[F'':F']=d_j(\tau\otimes F')$ realises $\ed_k(\tau\otimes F'')\le j$ then $[F'':F]=[F':F]\,d_j(\tau\otimes F')$. (c) is a restatement, (d) is (B2) applied to each $\tau\otimes F'$, (e) is (B3).
\end{proof}

The last jump degree is the only slot of the profile that is understood for every $G$.

\begin{proposition}[the last jump degree]\label{prop:d0}
Let $H=\im\varphi_\tau$ and $E$ the splitting field of $\tau$.
\begin{enumerate}[label=\textup{(\alph*)}]
\item $d_0(\tau)\le\min\{[F':F]:\tau\otimes F'\text{ split}\}=|H|$, with equality if $k$ is algebraically closed.
\item Let $k$ be perfect and $E/k$ regular. If $[F':F]<\infty$ and $\tau\otimes F'$ descends to the algebraic closure $k_1$ of $k$ in $F'$, then $[k_1:k]\le[F':F]$ is finite and $[F':F]\ge|H|\cdot[k_1:k]$. In particular $d_0(\tau)=|H|$.
\item For $\tgen$ over a perfect $k$: $d_0(\tgen)=|G|$, and $\ed_k(\tgen\otimes F')=0$ forces $[F':K]\ge|G|\,[k_1:k]$.
\end{enumerate}
\end{proposition}
\begin{proof}
(a) $T\otimes_FF'$ is split iff $E$ embeds $F$-linearly into $F'$, which forces $[F':F]\ge|H|$; $F'=E$ works; for $k$ algebraically closed use (B3).
(b) First, $[k_1:k]\le[F':F]$: as $F\subseteq E$ and $E/k$ is regular, $k$ is algebraically closed in $F$, so for every finite $k'/k$ inside $k_1$ the ring $k'\otimes_kF$ is a field, equal to $k'F\subseteq F'$, of degree $[k':k]$ over $F$. Now let $\tau\otimes F'\cong\tau_0\otimes_{k_1}F'$ with $\tau_0$ over $k_1$, of splitting field $l/k_1$ Galois with group $H_0$. The total spaces of $\tau\otimes F'$ and $\tau_0\otimes_{k_1}F'$ are isomorphic $F'$-algebras, each a product of copies of one field, namely a compositum $EF'$ (for some $F$-embedding of $E$ into a fixed algebraic closure $\Omega$ of $F'$), respectively $lF'$ (for some $k_1$-embedding of $l$). An $F'$-isomorphism $EF'\to lF'$ extends to an $F'$-automorphism of $\Omega$; it fixes $k_1\subseteq F'$, so after replacing $l$ by its image we may assume $EF'=lF'$. Since $l/k$ is finite separable and $E/k$ regular, $l\otimes_kE$ is a field and $[lE:E]=[l:k]=|H_0|[k_1:k]$; as $lE\subseteq EF'$,
\[
|H_0|[k_1:k]\,|H|=[lE:E][E:F]\le[EF':F]=[lF':F'][F':F]\le|H_0|\,[F':F].
\]
(c) $E=L$ is regular over $k$ and $H=G$.
\end{proof}

\begin{remark}
Without regularity $d_0<|H|$ is possible ($k=\Q$, $\tau=\Q(t)(\sqrt2)$ is constant, $d_0=1$). For $k$ algebraically closed, $\edd d_k(\tau)\ge1$ for $d<|H|$. For abelian $G$ with $\operatorname{char}k\nmid|G|$ the equality $d_0(\tgen)=|G|$ of (c) holds without assuming $k$ perfect (Theorem~\ref{thm:rankindex} with $A=G$: $\rho_G(d)\ge1$ for $d<|G|$).
\end{remark}

\begin{lemma}[purely inseparable invariance]\label{lem:pi}
Let $\operatorname{char}k=p$, $k$ perfect, $G$ any finite group, and $F'/F$ finite purely inseparable. Then $\ed_k(\tau\otimes_FF')=\ed_k(\tau)$. Consequently $\edd d_k(\tau)$ is the minimum over separable $F'/F$ of degree $\le d$.
\end{lemma}
\begin{proof}
Choose $m$ with $F'\subseteq F^{1/p^m}$ (inside an algebraic closure). Then $\ed_k(\tau\otimes F^{1/p^m})\le\ed_k(\tau\otimes F')\le\ed_k(\tau)$, and it suffices to prove $\ed_k(\tau\otimes F^{1/p^m})\ge\ed_k(\tau)$. Let $\sigma\colon F^{1/p^m}\to F$, $x\mapsto x^{p^m}$; it is a field isomorphism with $\sigma(k)=k$ since $k$ is perfect. For any $G$-torsor $\tau'$ over $F^{1/p^m}$ let $\sigma_*\tau'$ be the torsor over $F$ obtained by transport of structure. The map $F_0\mapsto\sigma(F_0)$ is a bijection between subfields of $F^{1/p^m}$ containing $k$ and subfields of $F$ containing $k$, preserving $\trdeg_k$, and $\tau'$ descends to $F_0$ iff $\sigma_*\tau'$ descends to $\sigma(F_0)$; so $\ed_k(\sigma_*\tau')=\ed_k(\tau')$. Finally $\sigma_*(\tau\otimes F^{1/p^m})\cong\tau$: the $p^m$-power map extends to $F^{1/p^m}F^{\rm sep}\to F^{\rm sep}$ and commutes with all automorphisms, so the induced isomorphism $\Gam_{F^{1/p^m}}\to\Gam_F$ is the canonical one (restriction along $F\subseteq F^{1/p^m}$, an isomorphism of absolute Galois groups), and the classifying homomorphism of $\sigma_*(\tau\otimes F^{1/p^m})$ is $\varphi_\tau$. The last statement: any finite $F'/F$ is purely inseparable over its separable closure $F'_s$ in $F'$, $[F'_s:F]\le[F':F]$, and $\ed_k(\tau\otimes F')=\ed_k(\tau\otimes F'_s)$.
\end{proof}

The standard statement is the subgroup inequality $\ed(H)\le\ed(G)$ \cite[Thm.~6.19]{BF03}, \cite[Prop.~3.15]{Mer13}; we need the following torsor-wise equality, over an arbitrary base field $k$, which implies it (take $\tau_H$ generic).

\begin{lemma}[induced torsors]\label{lem:ind}
For finite groups $H\le G$ and an $H$-torsor $\tau_H$ over $F$, $\ed_k(\Ind_H^G\tau_H)=\ed_k(\tau_H)$.
\end{lemma}
\begin{proof}
``$\le$'': if $\tau_H$ is defined over $F_0\subseteq F$ by an $H$-torsor $\tau_0$, then $\Ind_H^G\tau_0$ is a $G$-torsor over $F_0$ that becomes $\Ind_H^G\tau_H$ over $F$.

``$\ge$''. Step 1: we may assume $T_H$ is a field. Write $T_H$ for the total space of $\tau_H$; it is a product of fields permuted transitively by $H$, so $\tau_H=\Ind_{H'}^H\tau'$ with $\tau'$ an $H'$-torsor whose total space $T'$ is a field. Then $\Ind_H^G\tau_H=\Ind_{H'}^G\tau'$ and, by ``$\le$'', $\ed_k(\tau_H)\le\ed_k(\tau')$; so it suffices to show $\ed_k(\Ind_{H'}^G\tau')\ge\ed_k(\tau')$.

Step 2: a descent of $\Ind_H^G\tau_H$ and its stabiliser. Let $T_0\to\Spec F_0$ be a $G$-torsor, $F_0\subseteq F$, with $\trdeg_kF_0=\ed_k(\Ind_H^G\tau_H)$ and $T_0\otimes_{F_0}F\cong\Ind_H^G T_H$. Let $T_0'$ be a connected component of $T_0$, a field, with stabiliser $H_0\le G$. The components of $T_0'\otimes_{F_0}F$ are among the $gT_H$ (the components of $\Ind_H^GT_H$) and are permuted transitively by $H_0$. After replacing $H$ by a conjugate $gHg^{-1}$ and $T_H$ by $gT_H$ (which does not change $\Ind_H^GT_H$ up to isomorphism) we may therefore assume $T_H\subseteq T_0'\otimes_{F_0}F$ is one of them, and then $H\le H_0$, since $H$ stabilises $T_H$.

Step 3: $\tau_H$ descends to $F_1:=(T_0')^H$. The extension $F_1/F_0$ is separable ($T_0'$ is \'etale over $F_0$), so $F_1\otimes_{F_0}F=\prod_iF_i$ is a product of fields. The factor $T_0'\otimes_{F_1}F_{i_0}$ through which $T_H$ factors is an $H$-torsor over $F_{i_0}$ whose total space is a product of fields permuted transitively by $H$; $T_H$ is $H$-stable and is one of these components, so $T_H$ is that factor. Now $F_{i_0}$ is an $F$-algebra through the second factor of $F_1\otimes_{F_0}F$, so $F\subseteq F_{i_0}$; and $F_{i_0}\subseteq T_H$ is fixed by $H$ (which acts $F_1$-linearly, $F_1=(T_0')^H$), so $F_{i_0}\subseteq T_H^H=F$. Hence $F_{i_0}=F$. Thus $\tau_H$ is defined over $F_1\subseteq T_0'$, and $\trdeg_kF_1=\trdeg_kT_0'=\trdeg_kF_0$, because $T_0'/F_1$ is finite. Hence $\ed_k(\tau_H)\le\ed_k(\Ind_H^G\tau_H)$.
\end{proof}

\begin{corollary}[subgroup profile]\label{cor:sub}
For every subgroup $H'\le G$: $\edd{[G:H']}_k(\tgen)\le\ed_k(\tgen\otimes L^{H'})=\ed_k(H')$. Hence $d_j(\tgen)\le\min\{[G:H']:\ed_k(H')\le j\}$.
\end{corollary}
\begin{proof}
$\tgen\otimes L^{H'}=\Ind_{H'}^G(L/L^{H'})$, and $L/L^{H'}$ is the generic $H'$-torsor of $V|_{H'}$ over $k$; apply Lemma~\ref{lem:ind}.
\end{proof}

\begin{proposition}[base-field extension bound]\label{prop:bfe}
Let $k'/k$ be finite of degree $d$. Then $\edd d_k(\tau)\le\ed_{k'}(G)$. More precisely, every field factor $F'$ of $F\otimes_kk'$ has $[F':F]\le d$ and
\[
\ed_k(\tau\otimes F')=\ed_{k'}(\tau\otimes F')\le\ed_{k'}(G),
\]
with equality on the left because $k'/k$ is algebraic.
\end{proposition}
\begin{proof}
$F'\supseteq k'$ and $[F':F]\le d$. A descent $\tau\otimes F'\cong\tau_0\otimes_{F_0}F'$ with $k'\subseteq F_0\subseteq F'$ and $\trdeg_{k'}F_0\le\ed_{k'}(G)$ is also a descent over $k$, and $\trdeg_kF_0=\trdeg_{k'}F_0$ since $k'/k$ is algebraic; the equality is (B2).
\end{proof}

\begin{corollary}[roots-of-unity defect]\label{cor:rou}
Let $e$ be the exponent of $G$, $\operatorname{char}k\nmid e$, and $d_e=[k(\mu_e):k]$. For every $G$-torsor $\tau$ over every $F\supseteq k$,
\[
\edd{d_e}_k(\tau)\le\edd{[F(\mu_e):F]}_k(\tau)\le\ed_{k(\mu_e)}(G).
\]
In particular: (i) if $G$ is abelian, $\edd{d_e}_k(\tau)\le\rk G$ (Lemma~\ref{lem:kummer} over $k(\mu_e)$); (ii) if $G$ is a $p$-group and $\operatorname{char}k\ne p$, $\edd{p-1}_k(\tau)\le\edd{[k(\mu_p):k]}_k(\tau)\le\edd{[F(\mu_p):F]}_k(\tau)\le\ed_{k(\mu_p)}(G)$, the minimal dimension of a faithful representation of $G$ over $k(\mu_p)$ (Theorem~\ref{thm:KM}).
\end{corollary}
\begin{proof}
Proposition~\ref{prop:bfe} with $k'=k(\mu_e)$ (resp.\ $k'=k(\mu_p)$ in (ii)): $F(\mu_e)$ is a field factor of $F\otimes_kk(\mu_e)$, of degree $[F(\mu_e):F]\le d_e$ over $F$, so $\edd{[F(\mu_e):F]}_k(\tau)\le\ed_k(\tau\otimes F(\mu_e))\le\ed_{k(\mu_e)}(G)$, and the middle inequality is monotonicity. The values in (i), (ii) are Lemma~\ref{lem:kummer} and Theorem~\ref{thm:KM}.
\end{proof}

\begin{remark}\label{rem:rou}
For $\tgen$ and $F'=K(\mu_e)=Kk(\mu_e)$ the torsor $\tgen\otimes F'$ is the generic torsor of $V$ over $k(\mu_e)$ (Lemma~\ref{lem:regular}(c), proved in Section~\ref{sec:sylow}), so Proposition~\ref{prop:bfe} gives $\ed_k(\tgen\otimes K(\mu_e))=\ed_{k(\mu_e)}(\tgen\otimes K(\mu_e))=\ed_{k(\mu_e)}(G)$: the bound of Corollary~\ref{cor:rou} for $\tgen$ is realised by the extension $K(\mu_e)$. It is not optimal in general. Corollary~\ref{cor:rou} is the case $k'=k(\mu_e)$ of Proposition~\ref{prop:bfe}, and for $G=\Z/p$, $p$ odd, the smaller field $k'=k(\zeta_p+\zeta_p^{-1})$ already gives $\ed_{k'}(\Z/p)\le1$ (Lemma~\ref{lem:B1}, proved in Section~\ref{sec:thmB}); this halves the degree whenever $\zeta_p\notin k(\zeta_p+\zeta_p^{-1})$. Exactly when Corollary~\ref{cor:rou} is sharp for $\Z/p$ in characteristic $0$ is determined in Corollary~\ref{cor:sharp}, with one family of cases left open.
\end{remark}

\begin{example}[$\Z/n$]\label{ex:cyc}
Let $\mu_n\subset k$, $\operatorname{char}k\nmid n$. The generic torsor $K(\sqrt[n]a)/K$, $K=k(a)$, has $\ed_k=1$ and $d_0=n$ (Proposition~\ref{prop:d0} if $k$ is perfect, Corollary~\ref{cor:A7} in general), so its profile is $(1,\dots,1,0)$.
\end{example}

% ---- end profile.tex ----

% ---- begin geom.tex ----
\section{Geometric meaning of the profile}\label{sec:geom}

\begin{theorem}[geometric meaning of the profile]\label{thm:geom}
Let $G$ be a finite group, $k$ any field, $\tau\colon T\to\Spec F$ a $G$-torsor, $F\supseteq k$, and $m,d\ge0$. The following are equivalent:
\begin{enumerate}[label=\textup{(\arabic*)}]
\item $\edd d_k(\tau)\le m$;
\item there is a quasi-projective generically free primitive $G$-variety $Z$ over $k$ with $\dim Z\le m$ such that $\tw{\tau}{Z}$ has a closed point of degree $\le d$ in the free locus $\tw{\tau}{(Z^{\free})}$;
\item there are $Z$ as in \textup{(2)} and a field $F'/F$ with $[F':F]\le d$ such that $\tw{\tau\otimes F'}{Z}$ has a free $F'$-point, i.e.\ a $G$-equivariant morphism $T\times_FF'\to Z^{\free}$.
\end{enumerate}
For $\tau=\tgen$ (generic torsor of $V$), \textup{(3)} says: there is a $G$-equivariant rational map $W'\dashrightarrow Z^{\free}$ from the normalisation $W'$ of $V$ in $L\otimes_KF'$, a (possibly reducible) $G$-variety dominating $V$ with degree $[F':K]\le d$; when $L\otimes_KF'$ is a field this is a $G$-equivariant correspondence $V\rightsquigarrow Z$ of degree $\le d$ in the sense of \cite[(2.6)]{RS22} landing in $Z^{\free}$, and $Z$ can then be taken irreducible with faithful $G$-action. In particular, by (2), $d_j(\tau)$ is the least degree of a closed point in the free locus of $\tw{\tau}{Z}$ over all such $Z$ with $\dim Z\le j$.
\end{theorem}
\begin{proof}
(2)$\Leftrightarrow$(3): a closed point $x$ of $\tw{\tau}{Z}$ with residue field $F'=\kappa(x)$ is an $F'$-point of $\tw{\tau}{Z}\otimes_FF'=\tw{\tau\otimes F'}{Z}$, free iff $x$ lies in the open free locus (which commutes with base change); conversely the image of $\Spec F'\to\tw{\tau\otimes F'}{Z}\to\tw{\tau}{Z}$ is a point with residue field contained in $F'$, hence finite over $F$, hence closed, of degree $\le[F':F]$, and free if the $F'$-point is. The identification of free $F'$-points with equivariant maps $T_{F'}\to Z^{\free}$ is Fact~\ref{fact:rs}(a).

(3)$\Rightarrow$(1): Lemma~\ref{lem:freepoint} gives $\ed_k(\tau\otimes F')\le\dim Z\le m$.

(1)$\Rightarrow$(3): let $[F':F]\le d$ and $\tau\otimes F'\cong\tau_0\otimes_{F_0}F'$ with $\trdeg_kF_0\le m$. Replacing $F_0$ by the subfield generated over $k$ by the finitely many coefficients that define $\tau_0$ and the isomorphism $\tau\otimes F'\cong\tau_0\otimes_{F_0}F'$, we may assume $F_0$ finitely generated over $k$. Spread $\tau_0$ out to a $G$-torsor $Z\to B$ over an integral affine $k$-variety $B$ with $k(B)=F_0$, $\dim B\le m$. Then $Z$ is quasi-projective, $G$ acts freely (so $Z^{\free}=Z$) and $Z$ is primitive (its generic fibre $T_0$ is a torsor over the field $F_0$, on whose components $G$ acts transitively). The composite $T_{F'}\to T_0\to Z$ is $G$-equivariant.

For $\tgen$: $F'$ is a finite extension of $K=k(V/G)$, so it is the function field of a variety $B'$ with a finite dominant map $B'\to V/G$ of degree $[F':K]$, and $W':=$ normalisation of $V\times_{V/G}B'$ is a $G$-variety (action on the first factor) with $k(W')=L\otimes_KF'$ (a product of fields) and $W'\to V$ dominant of degree $[F':K]$. An equivariant $\Spec(L\otimes_KF')\to Z^{\free}$ is an equivariant rational map $W'\dashrightarrow Z^{\free}$. If $L\otimes_KF'$ is a field, $T_{F'}$ is connected, hence so is $T_0$, so $\tau_0$ is a field and $Z$ is irreducible with faithful action.

The ingredients are: the existence of the twist $\tw{\tau}{Z}=(T\times_kZ)/G$ as a quasi-projective $F$-scheme compatible with base change and open immersions \cite[\S3]{DR15}; the identification of $F$-points of $\tw{\tau}{Z}$ with equivariant maps $T\to Z_F$ (elementary for finite $G$); the fact that $Z^{\free}\to Z^{\free}/G$ is a $G$-torsor \cite[Exp.~V, \S\S1--2]{SGA1}; and spreading out. None of these depends on the characteristic or on $G$.
\end{proof}

\begin{remark}[relation with \cite{RS22}, Theorem~10.1]\label{rem:rs101}
Theorem~10.1 of \cite{RS22} says: for generically free primitive $G$-varieties $X$, $Y$ with $X$ smooth, a $G$-equivariant correspondence $Y\rightsquigarrow X$ of degree prime to $q$ gives $\ed_{k,q}(X)\ge\ed_{k,q}(Y)$; the proof passes through the $q$-closure, which is large, and Fact~\ref{fact:rs}(b). In the language of Theorem~\ref{thm:geom}, a free closed point of degree $d$ on $\tw{\tgen}{Z}$ is a correspondence $V\rightsquigarrow Z$ of degree $d$ into the free locus. The prime-to-$q$ part of the profile is thus governed by \cite[Thm.~10.1]{RS22} (Theorem~\ref{thm:div}); the $q$-power part is not.
\end{remark}

The next proposition shows that the profile of the generic torsor is the universal profile.

\begin{proposition}[the generic torsor dominates every profile]\label{prop:vers}
Let $G\ne1$, $k$ be any field, $V$ a faithful representation of $G$ over $k$, and $\tau$ a $G$-torsor over a field $F\supseteq k$. Then for every $d\ge1$ and every prime $q$
\[
\edd d_k(\tau)\le\edd d_k(\tgen),\qquad
\edd d_k{}'(\tau)\le\edd d_k{}'(\tgen)
\]
(the second for the prime-to-$q$ profiles), hence $d_j(\tau)\le d_j(\tgen)$ and $d_j^{(q')}(\tau)\le d_j^{(q')}(\tgen)$ for all $j$. Consequently $\sup_\tau\edd d_k(\tau)=\edd d_k(\tgen)$, and the profile of $\tgen$ does not depend on the faithful representation $V$.
\end{proposition}
\begin{proof}
If $F$ is finite, then $F$ is algebraic over $k$, so $\ed_k(\tau\otimes F'')=0$ for every $F''/F$ by (B3) and there is nothing to prove; let $F$ be infinite. Let $F'/K$ realise $\edd d_k(\tgen)=m$ (resp.\ $\edd d_k{}'(\tgen)=m$), with $[F':K]=d'\le d$ (resp.\ $q\nmid d'$). By Theorem~\ref{thm:geom} there are a quasi-projective generically free primitive $G$-variety $Z$ of dimension $\le m$ and a $G$-equivariant morphism $\Spec{}(L\otimes_KF')\to Z^{\free}$; note that $L\otimes_KF'$ is reduced because $L/K$ is separable, whether or not $F'/K$ is. Let $W'$ be the normalisation of $V$ in $L\otimes_KF'$, i.e.\ $\Spec$ of the integral closure of $k[V]$ in $L\otimes_KF'$. Then $\pi\colon W'\to V$ is finite, $G$-equivariant and generically of degree $d'$, and the equivariant map $W'\dashrightarrow Z^{\free}$ is defined on a dense $G$-stable open $U$. Choose a dense $G$-stable open $V_1\subseteq V$ such that $\pi^{-1}(V_1)\subseteq U$ and $\pi$ is finite and flat of degree $d'$ over $V_1$ (generic flatness; $\pi$ need not be \'etale). To see that $V_1$ may be taken $G$-stable: $\pi(W'\smallsetminus U)$ is a $G$-stable proper closed subset of $V$, since $\pi$ is finite and equivariant and $W'\smallsetminus U$ has dimension $<\dim V$; by generic freeness $\pi_*\mathcal O_{W'}$ is free over $V_f$ for some $f\ne0$, hence over every $gV_f$ by equivariance; take $V_1=V\smallsetminus\bigl(\pi(W'\smallsetminus U)\cup\bigcup_{g\in G}V(gf)\bigr)$. Twisting by $\tau$ (functoriality of twists, \cite[\S3]{DR15}) gives a finite flat $F$-morphism $\tw{\tau}{\pi^{-1}(V_1)}\to\tw{\tau}{V_1}$ of degree $d'$ and an $F$-morphism $\tw{\tau}{\pi^{-1}(V_1)}\to\tw{\tau}{(Z^{\free})}$. Now $\tw{\tau}{V}$ is an $F$-form of the vector space $V_F$, hence isomorphic to affine space by Hilbert~90, so $\tw{\tau}{V_1}(F)\ne\emptyset$ as $F$ is infinite; pick $x\in\tw{\tau}{V_1}(F)$. The fibre over $x$ is $\Spec{}A$ for a finite $F$-algebra $A=\prod_iA_i$ of dimension $d'$, with $A_i$ local of length $\ell_i$ and residue field $F_i$, so $\sum_i\ell_i[F_i:F]=d'$. Each $\Spec{}F_i\to\tw{\tau}{(Z^{\free})}$ has as image a closed point of $\tw{\tau}{Z}$ in the free locus whose degree divides $[F_i:F]\le d'\le d$ (its residue field is a subfield of $F_i$); so $\edd d_k(\tau)\le m$ by Theorem~\ref{thm:geom}. If $q\nmid d'$, then some $\ell_i[F_i:F]$ is prime to $q$, hence so is $[F_i:F]$ and the degree of the corresponding point, which gives $\edd d_k{}'(\tau)\le m$ (with $F''=\kappa(y_i)$ the field of definition of that point). The last two statements: the supremum is attained at $\tgen$, and generic torsors for two faithful representations $V,V'$ are torsors over infinite fields, so each dominates the other.
\end{proof}
\begin{remark}
The only property of $V$ used is that $\tw{\tau}{V_1}(F)\ne\emptyset$ for every torsor $\tau$ over an infinite field $F\supseteq k$ and every dense open $V_1\subseteq V$; for a linear representation this is Hilbert~90.
\end{remark}

% ---- end geom.tex ----

% ---- begin sylow.tex ----
\section{Sylow divisibility and essential dimension at a prime}\label{sec:sylow}

\begin{lemma}[$q$-closure]\label{lem:qcl}
Let $q$ be a prime, $S\le\Gam_F$ a Sylow pro-$q$ subgroup and $F^{(q)}=(F^{\rm sep})^S$. Then $F^{(q)}$ is a union of finite subextensions of degree prime to $q$, every finite separable extension of $F^{(q)}$ has $q$-power degree, and $\ed_{k,q}(\tau)=\ed_k(\tau\otimes F^{(q)})$ \cite[(2.5)]{RS22}. If $F'/F$ is finite separable with $q\nmid[F':F]$, then there is an $F$-embedding $F'\hookrightarrow F^{(q)}$, and hence $\ed_k(\tau\otimes F')\ge\ed_{k,q}(\tau)$.
\end{lemma}
\begin{proof}
The first three statements are quoted. $F'\otimes_FF^{(q)}$ is a product of finite separable extensions of $F^{(q)}$, each of $q$-power degree, whose degrees sum to $[F':F]$; if none had degree $1$ the sum would be divisible by $q$. So some factor is $F^{(q)}$, giving the embedding, and (B1) gives the inequality. (For a given embedding $F'\subseteq F^{\rm sep}$ the compositum $F'F^{(q)}$ need not be $F^{(q)}$: $F=\Q$, $q=2$, $F'=\Q(\sqrt[3]2)$ has $F'\otimes_\Q\Q^{(2)}\cong\Q^{(2)}\times\Q^{(2)}(\sqrt{-3})$.)
\end{proof}

\begin{lemma}[prime-to-$q$ extensions of the base field]\label{lem:qinv}
Let $q$ be a prime and $k'/k$ a finite separable extension with $q\nmid[k':k]$. Then $\ed_{k',q}(G)=\ed_{k,q}(G)$ (compare \cite[Prop.~2.4]{Mer13}, which has the corresponding statement for ed at a prime under base change; we give the short proof for completeness).
\end{lemma}
\begin{proof}
For a $G$-torsor $\tau$ over a field $F\supseteq k'$, $\ed_{k,q}(\tau)$ and $\ed_{k',q}(\tau)$ are minima over the same set of extensions $F'/F$ of $\ed_k(\tau\otimes F')$ and $\ed_{k'}(\tau\otimes F')$, which agree by (B2) since $k'/k$ is algebraic; so $\ed_{k,q}(\tau)=\ed_{k',q}(\tau)$, and taking suprema, $\ed_{k',q}(G)\le\ed_{k,q}(G)$. Conversely let $\tau$ be a $G$-torsor over a field $F\supseteq k$. The algebra $F\otimes_kk'$ is \'etale of dimension $[k':k]$, so it has a field factor $F_1$ with $q\nmid[F_1:F]$; $F_1/F$ is separable and $F_1\supseteq k'$. Every finite separable $F''/F_1$ with $q\nmid[F'':F_1]$ is finite separable over $F$ with $q\nmid[F'':F]$, so $\ed_{k,q}(\tau)\le\ed_{k,q}(\tau\otimes F_1)=\ed_{k',q}(\tau\otimes F_1)\le\ed_{k',q}(G)$. Hence $\ed_{k,q}(G)\le\ed_{k',q}(G)$.
\end{proof}

\begin{proposition}\label{prop:genq}
For every prime $q$, $\ed_{k,q}(\tgen)=\ed_{k,q}(G)$.
\end{proposition}
\begin{proof}
``$\le$'' is the definition. For ``$\ge$'' let $e=\ed_{k,q}(\tgen)$, realised by a finite separable extension $K'/K$ with $q\nmid[K':K]$ and $\ed_k(\tgen\otimes K')=e$. Let $\tau$ be a $G$-torsor over a field $F\supseteq k$. If $F$ is finite, $F$ is algebraic over $k$ and $\ed_k(\tau)=0$. If $F$ is infinite, the proof of Proposition~\ref{prop:vers}, applied with $F'=K'$ (which is separable over $K$, so that $\pi$ may be chosen \'etale over a dense $G$-stable open $V_1$: the \'etale locus of $\pi$ is open, dense and $G$-stable, $\pi$ is finite, and one shrinks as in the proof of Proposition~\ref{prop:vers}), produces a quasi-projective generically free primitive $G$-variety $Z$ with $\dim Z\le e$ and an \'etale $F$-algebra $\prod_iF_i$ of dimension $[K':K]$ such that $\tw{\tau\otimes F_i}{Z}$ has a free $F_i$-point for each $i$. Some $F_i$ has $q\nmid[F_i:F]$; it is separable over $F$, and $\ed_k(\tau\otimes F_i)\le\dim Z\le e$ by Lemma~\ref{lem:freepoint}. Hence $\ed_{k,q}(\tau)\le e$ for every $\tau$.
\end{proof}

\begin{theorem}[Sylow divisibility]\label{thm:div}
Let $k$ be perfect (for instance of characteristic $0$). Let $\tau$ be a $G$-torsor over $F$.
\begin{enumerate}[label=\textup{(\alph*)}]
\item For every prime $q$ and every $j<\ed_{k,q}(\tau)$, $q$ divides $d_j(\tau)$: every $F'/F$ with $\ed_k(\tau\otimes F')\le j$ has $q\mid[F':F]$.
\item Every prime $q$ with $\ed_{k,q}(G)>j$ divides $d_j(\tgen)$; hence $d_j(\tgen)\ge\prod_{q:\,\ed_{k,q}(G)>j}q$.
\item For every $d\ge1$,
$\displaystyle\edd d_k(\tgen)\ \ge\ \min_{1\le d'\le d}\ \max_{q\nmid d'}\ed_{k,q}(G)\ \ge\ \max_{q>d}\ed_{k,q}(G)$ ($q$ prime).
\item The prime-to-$q$ profile $\edd d_k{}'(\tau)=\min\{\ed_k(\tau\otimes F'):[F':F]\le d,\ q\nmid[F':F]\}$ is non-increasing with terminal value $\ed_{k,q}(\tau)$, reached at a finite prime-to-$q$ degree; for $\tgen$ the terminal value is $\ed_{k,q}(G)$.
\end{enumerate}
\end{theorem}
\begin{proof}
(a) Let $q\nmid[F':F]$. By Lemma~\ref{lem:pi} (in characteristic $p>0$) we may replace $F'$ by its separable closure $F'_s$ over $F$, whose degree divides $[F':F]$; then Lemma~\ref{lem:qcl} gives $\ed_k(\tau\otimes F')\ge\ed_{k,q}(\tau)>j$. (b) adds Proposition~\ref{prop:genq}. (c) A witness for $\edd d_k(\tgen)$ has some degree $d'\le d$, and (a) and (b) apply to every $q\nmid d'$. (d) By the argument of (a), $\edd d_k{}'(\tau)\ge\ed_{k,q}(\tau)$ for all $d$; by definition $\ed_{k,q}(\tau)$ is attained by a finite separable $F'/F$ of degree prime to $q$, so $\edd d_k{}'(\tau)=\ed_{k,q}(\tau)$ for $d\ge[F':F]$. The last statement is Proposition~\ref{prop:genq}.
\end{proof}

\begin{remark}
In the language of Theorem~\ref{thm:geom} and Remark~\ref{rem:rs101}, (a) is Reichstein--Scavia's Theorem~10.1 applied to the correspondence of degree prime to $q$ furnished by a free closed point; the proof above via the $q$-closure is the same mechanism.
\end{remark}

For the generic torsor the field factor in Proposition~\ref{prop:bfe} is unique, and base change along a constant-field extension produces the generic torsor over the larger field.

\begin{lemma}[$K$ is regular over $k$; base change of $\tgen$]\label{lem:regular}
Let $V$ be a faithful representation of $G$ over $k$, $L=k(V)$, $K=L^G$, all inside an algebraically closed field $\Om$.
\begin{enumerate}[label=\textup{(\alph*)}]
\item $K/k$ is regular.
\item For every algebraic extension $k'/k$ inside $\Om$, $K\otimes_kk'$ is a field, so $[Kk':K]=[k':k]$ when $k'/k$ is finite; and $Kk'=k'(V)^G$.
\item $L\otimes_KKk'\cong Lk'=k'(V)$ as $G$-algebras. Hence $\tgen\otimes_KKk'$ is the generic torsor of $V_{k'}$ over $k'$.
\end{enumerate}
\end{lemma}
\begin{proof}
(a) $L=k(x_1,\dots,x_n)$ is purely transcendental, hence regular over $k$; subextensions of regular extensions are regular \cite[VIII~\S4]{La}.

(b) A regular extension is linearly disjoint from every algebraic extension \cite[VIII, Thm.~4.12]{La}, so $K\otimes_kk'$ is a field, of degree $[k':k]$ over $K$ when this is finite. For the second claim: $Lk'$ is generated by $k'$ and the $x_i$, which remain algebraically independent over $k'$ for transcendence-degree reasons, so $Lk'=k'(V)$. Since $G$ acts $k'$-linearly, $Kk'\subseteq(Lk')^G$. Now $Lk'=L\cdot Kk'$ gives $[Lk':Kk']\le[L:K]=|G|$, while $[Lk':(Lk')^G]=|G|$ by Artin's theorem. Hence $Kk'=(Lk')^G$.

(c) The natural map $L\otimes_KKk'\to Lk'$ is surjective and $G$-equivariant, and both sides have dimension $|G|$ over $Kk'$.
\end{proof}

\subsection{$q$-groups and minimal faithful representations}

\begin{proposition}[$q$-groups]\label{prop:qgp}
Let $G$ be a $q$-group of order $q^n$, $\operatorname{char}k\ne q$ ($k$ perfect if positive), $c=[k(\mu_q):k]$. Then:
\begin{enumerate}[label=\textup{(\alph*)}]
\item $\edd d_k(\tgen)=\ed_{k,q}(G)$ for $c\le d\le q-1$, and $\ed_{k,q}(G)\le\edd d_k(\tgen)\le\ed_k(G)$ for $d<c$;
\item $\ed_{k,q}(G)=\ed_{k(\mu_q)}(G)$, and this is the minimal dimension of a faithful $k(\mu_q)$-representation of $G$;
\item $q\mid d_j(\tgen)$ for every $j<\ed_{k,q}(G)$, and $d_0(\tgen)=q^n$;
\item if $\mu_q\subset k$: $\edd d_k(\tgen)=\ed_k(G)$ for $d<q$ and all steps of the profile occur at multiples of $q$, i.e.\ $\edd d_k(\tgen)=\edd{q\lfloor d/q\rfloor}_k(\tgen)$ for $d\ge q$.
\end{enumerate}
\end{proposition}
\begin{proof}
(b) The extension $k(\mu_q)/k$ is separable of degree $c\mid q-1$, so $\ed_{k,q}(G)=\ed_{k(\mu_q),q}(G)$ by Lemma~\ref{lem:qinv}, and $\ed_{k(\mu_q),q}(G)=\ed_{k(\mu_q)}(G)=\rdim_{k(\mu_q)}(G)$ by Theorem~\ref{thm:KM} over $k(\mu_q)$. (a) Lower bounds: every degree $<q$ is prime to $q$, so Theorem~\ref{thm:div}(c) gives $\edd d_k(\tgen)\ge\ed_{k,q}(G)$ for $d\le q-1$. Upper bound for $d\ge c$: Corollary~\ref{cor:rou} and (b). (c) is Theorem~\ref{thm:div}(b) and Proposition~\ref{prop:d0}(c). (d) With $c=1$, (a) gives $\edd d_k(\tgen)=\ed_{k,q}(G)$ for $1\le d\le q-1$, and $\edd1_k(\tgen)=\ed_k(G)$; so $\ed_{k,q}(G)=\ed_k(G)$, every $d_j(\tgen)$ with $j<\ed_k(G)$ is a multiple of $q$ by (c), and Proposition~\ref{prop:shape}(c) gives the statement.
\end{proof}

In characteristic $p$ the prime-to-$p$ tail of a $p$-group collapses to $\le1$ \cite{RV18}; in characteristic $\ne q$ the prime-to-$q$ tail of a $q$-group is constant, equal to $\ed_{k,q}(G)$.

\label{ss:rdim}

For a finite group $G$ and a field $k$ let $\rdim_k(G)$ be the minimal dimension of a faithful $k$-representation of $G$. Always $\ed_k(G)\le\rdim_k(G)$, since a faithful representation is a generically free versal $G$-variety \cite{BR97,BF03}. The theorem of Karpenko and Merkurjev turns this into an equality for $p$-groups in the presence of $\zeta_p$, and thereby identifies the plateau of Proposition~\ref{prop:qgp} with a representation-theoretic quantity.

\begin{theorem}[Karpenko--Merkurjev {\cite[Thm.~4.1]{KM}}]\label{thm:KM}
Let $G$ be a finite $p$-group and $k$ a field with $\operatorname{char}k\ne p$ containing a primitive $p$-th root of unity. Then
\[
\ed_k(G)=\ed_{k,p}(G)=\rdim_k(G).
\]
\end{theorem}
Here $\ed_{k,p}$ is defined with separable extensions of degree prime to $p$ (Section~\ref{sec:conv}); \cite{KM} allow all finite extensions of degree prime to $p$. This does not matter: the minimum over all such extensions is $\le$ the minimum over separable ones, which is $\le\ed_k(G)$ (take $F'=F$), and \cite{KM} show that the former already equals $\ed_k(G)$, so all three agree.

\begin{proposition}[the Karpenko--Merkurjev plateau]\label{prop:KMplateau}
Let $G$ be a $p$-group, $\operatorname{char}k\ne p$ ($k$ perfect if positive), and $c=[k(\mu_p):k]$, a divisor of $p-1$.
\begin{enumerate}[label=\textup{(\alph*)}]
\item For $c\le d\le p-1$: $\edd d_k(\tgen)=\rdim_{k(\mu_p)}(G)=\ed_{k,p}(G)$; the profile is constant on $[c,p-1]$.
\item For $1\le d<c$: $\rdim_{k(\mu_p)}(G)\le\edd d_k(\tgen)\le\ed_k(G)$, with equality on the right for $d=1$.
\item If $\zeta_p\in k$: $\edd d_k(\tgen)=\rdim_k(G)=\ed_k(G)$ for $1\le d\le p-1$, and every drop of the profile occurs at a multiple of $p$: $\edd d_k(\tgen)=\edd{p\lfloor d/p\rfloor}_k(\tgen)$ for $d\ge p$.
\end{enumerate}
\end{proposition}
\begin{proof}
This is Proposition~\ref{prop:qgp}(a),(b),(d) with $q=p$, in which $\ed_{k(\mu_p)}(G)=\ed_{k(\mu_p),p}(G)=\ed_{k,p}(G)$ is rewritten as $\rdim_{k(\mu_p)}(G)$ by Theorem~\ref{thm:KM} over $k(\mu_p)$. Explicitly: every degree $\le p-1$ is prime to $p$, so Theorem~\ref{thm:div}(c) gives the lower bound $\ed_{k,p}(G)$ for all $d\le p-1$; Corollary~\ref{cor:rou}(ii) gives $\edd c_k(\tgen)\le\ed_{k(\mu_p)}(G)$; and the two values agree by Theorem~\ref{thm:KM} over $k(\mu_p)$ and Lemma~\ref{lem:qinv} applied to the separable extension $k(\mu_p)/k$ of degree $c$ prime to $p$. (c) is (a) with $c=1$ together with Proposition~\ref{prop:qgp}(c).
\end{proof}

\begin{remark}\label{rem:KMscope}
The lower bound of Theorem~\ref{thm:KM} is a bound on the canonical $p$-dimension of the classifying variety, and canonical $p$-dimension is insensitive to correspondences of degree prime to $p$; in the language of Theorem~\ref{thm:geom} and Remark~\ref{rem:rs101}, the Karpenko--Merkurjev theorem controls exactly the Sylow-governed, prime-to-$p$ part of the profile, which is the plateau of Proposition~\ref{prop:KMplateau}. The $p$-power part, beginning with $\edd p_k(\tgen)$, is beyond the method: a correspondence of degree $p$ can lower the canonical $p$-dimension. In (b), the profile can drop below $\ed_k(G)$ before degree $c$: for $G=\Z/p$ in characteristic $0$ with $c$ even and $c\ge4$, $\edd{c/2}_k(\tgen)=1<2\le\ed_k(\Z/p)$ (Corollary~\ref{cor:sharp}(ii)).
\end{remark}

\begin{proposition}[subgroup bound]\label{prop:subrdim}
For every subgroup $H\le G$ and every field $k$,
\[
\edd{[G:H]}_k(\tgen)\le\ed_k(H)\le\rdim_k(H).
\]
If $G$ is a $p$-group, $\operatorname{char}k\ne p$ and $\zeta_p\in k$, then for every $j\ge0$
\[
\edd{p^j}_k(\tgen)\ \le\ r_j(G,k):=\min\{\rdim_k(H):H\le G,\ [G:H]\le p^j\},
\]
and $\edd{p^j}_k(\tgen)\le\ed_k(H)=\rdim_k(H)$ is realised by the fixed field $L^H$ of any $H$ of index $\le p^j$.
\end{proposition}
\begin{proof}
The first chain is Corollary~\ref{cor:sub} and $\ed_k(H)\le\rdim_k(H)$. For the second, $[G:H]$ is a power of $p$, at most $p^j$, so Corollary~\ref{cor:sub} and monotonicity give $\edd{p^j}_k(\tgen)\le\ed_k(H)$, and $\ed_k(H)=\rdim_k(H)$ by Theorem~\ref{thm:KM} applied to the $p$-group $H$.
\end{proof}

\begin{example}[consistency with equality]\label{ex:rdim}
\begin{enumerate}[label=\textup{(\roman*)}]
\item $G=\Z/p^n$, $\operatorname{char}k\ne p$. For every such $k$, $\rdim_k(\Z/p^n)=[k(\zeta_{p^n}):k]$: a faithful representation has an irreducible summand on which the generator has order exactly $p^n$, and the irreducible representations with this property are the irreducible factors of $\Phi_{p^n}$ over $k$, all of dimension $[k(\zeta_{p^n}):k]$. For $\zeta_p\in k$ we have $\ed_k(\Z/p^n)=[k(\zeta_{p^n}):k]$: this is the case $G=\Z/p^n$ of Theorem~\ref{thm:KM}; Florence \cite{Fl08} proved it earlier under further conditions on the roots of unity in $k$. Now let $\zeta_p\in k$ and $[k(\zeta_{p^n}):k]=p^{n-1}$ (e.g.\ $k=\Q(\zeta_p)$ for $p$ odd, or $k=\Q$ for $p=2$); then every step of the tower $k\subset k(\zeta_{p^2})\subset\dots\subset k(\zeta_{p^n})$ has degree $p$, and
\[
\ed_k(\Z/p^n)=\rdim_k(\Z/p^n)=p^{n-1},\qquad \edd{p^j}_k(\tgen)\le p^{\,n-1-j}\quad(0\le j\le n-1),
\]
the bound being realised in two ways of the same degree $p^j$: by the fixed field $L^H$ of the index-$p^j$ subgroup $H=\langle g^{p^j}\rangle\cong\Z/p^{n-j}$, for which $\rdim_k(H)=[k(\zeta_{p^{n-j}}):k]=p^{n-1-j}$ (Proposition~\ref{prop:subrdim}), and by adjoining $\zeta_{p^{j+1}}$: $k'=k(\zeta_{p^{j+1}})$ has $[k':k]=p^j$ and $\ed_{k'}(\Z/p^n)=[k'(\zeta_{p^n}):k']=p^{n-1-j}$, so Proposition~\ref{prop:bfe} gives the same bound. By Proposition~\ref{prop:KMplateau}(c) the profile is constant $=p^{n-1}$ on $[1,p-1]$ and drops only at multiples of $p$: a staircase $p^{n-1},p^{n-2},\dots,1,0$ at the degrees $1,p,\dots,p^{n-1},p^n$ is consistent with everything known, and $d_0=p^n$ is Proposition~\ref{prop:d0}(c). For cyclic $G$ the rank--index theorem gives nothing beyond $d_0$ (Remark~\ref{rem:limits}).
\end{enumerate}
\end{example}

The subgroup formula. For a $p$-group $G$, $\operatorname{char}k\ne p$, $\zeta_p\in k$, Proposition~\ref{prop:subrdim} gives
\begin{equation}\label{eq:subgroupformula}
\edd d_k(\tgen)\ \le\ \min\{\rdim_k(H):H\le G,\ [G:H]\le d\}\qquad(d\ge1),
\end{equation}
with equality for $d<p$ and for $d\ge|G|$ (Propositions~\ref{prop:KMplateau}(c) and~\ref{prop:d0}(c)). Whether equality holds for all $d$ is open (Question~\ref{q:subgroupformula}). Equality holds for every $d$ when $G$ is abelian of exponent $p^a$ and $\mu_{p^a}\subset k$ (Corollary~\ref{cor:subgroupabelian}).

% ---- end sylow.tex ----

% ---- begin const.tex ----
\section{Constant classes over purely transcendental fields}\label{sec:const}

This section collects the geometric tools that turn a hypothetical witness for $\ed\le1$ over a field extension into a rigid object on a curve. Everything here is used in Section~\ref{sec:S5}. Throughout, ``curve'' means smooth projective curve over $k$.

\begin{lemma}[purely transcendental base change]\label{lem:constant}
\textup{(i)} Let $G$ be a finite group, $\tau$ a $G$-torsor over an infinite field $F\supseteq k$, and $t$ transcendental over $F$. Then $\ed_k(\tau\otimes_FF(t))=\ed_k(\tau)$. \textup{(ii)} Let $A$ be an abelian variety over a field $F$ (for instance an elliptic curve, or the Jacobian of a curve), or more generally a proper $F$-variety such that $A_{\bar F}$ is isomorphic to an abelian variety, and let $F'=F(t_1,\dots,t_m)$ be purely transcendental. Then $A(F')=A(F)$.
\end{lemma}
\begin{proof}
(i) $\le$ is \cite[Lemma~2.1(a)]{RS22}. Conversely let $\tau\otimes F(t)\cong\tau_0\otimes_{F_0}F(t)$ with $F_0\subseteq F(t)$, $\trdeg_kF_0=e$; replacing $F_0$ by the subfield generated over $k$ by the finitely many coefficients of equations defining $\tau_0$, we may assume $F_0$ finitely generated. Spread $\tau_0$ out to a $G$-torsor $Z\to B_0$ over an integral $k$-variety $B_0$ with $k(B_0)=F_0$, $\dim B_0=e$. The inclusion $F_0\subseteq F(t)$ is a dominant rational map $\A^1_F\dashrightarrow B_0$, defined on a dense open $U$, and the isomorphism of torsors over $F(t)$ spreads out to an isomorphism between the pullback of $Z$ and the constant torsor $\tau\times U'$ over a dense open $U'\subseteq U$. Since $F$ is infinite, $U'$ has an $F$-point $c$; restricting to $c$, $\tau$ is the pullback of $Z$ along $\Spec F\to B_0$, whose image point $b$ has $\trdeg_k\kappa(b)\le e$. Hence $\ed_k(\tau)\le e$.

(ii) By induction on $m$ (replacing $F$ by $F(t_1,\dots,t_{m-1})$) it suffices to treat $m=1$. A point of $A(F(t))$ is a rational map $\PP^1_F\dashrightarrow A$, which extends to a morphism $\phi\colon\PP^1_F\to A$ since $A$ is proper and $\PP^1_F$ is a regular curve. The base change $\phi_{\bar F}\colon\PP^1_{\bar F}\to A_{\bar F}$ is constant: otherwise its image would be a rational curve on the abelian variety $A_{\bar F}$, and an abelian variety contains no rational curves \cite[Ch.~I, Prop.~3.9]{Mil}.
Hence $\phi(\PP^1_F)$ is a single point $a$. As $\PP^1_F$ is reduced, $\phi$ factors through $\Spec\kappa(a)$, and an $F$-morphism $\PP^1_F\to\Spec\kappa(a)$ is an $F$-algebra map $\kappa(a)\to H^0(\PP^1_F,\mathcal O)=F$; so $\kappa(a)=F$ and $\phi$ is the constant map to the $F$-point $a$.
\end{proof}

\begin{lemma}[compression]\label{lem:compression}
Let $k=\bar k$, $L/K$ a $G$-Galois extension of fields containing $k$ with $G$ finite, $\tau=\Spec L$, and $F'/K$ finite with $\ed_k(\tau\otimes F')\le1$; let $N_0,H$ be as in Section~\ref{ss:tors} and assume $H\ne1$. Then there is a subfield $k\subseteq M\subseteq N_0$, stable under $H$, with $H$ acting faithfully on $M$ and $\operatorname{trdeg}_kM=1$.
\end{lemma}
\begin{proof}
Let $E\subseteq F'$, $\operatorname{trdeg}_kE\le1$, and $P_0/E$ a $G$-torsor with $P_0\otimes_EF'\cong\prod_{G/H}N_0$; the induced map $j:P_0\to\prod_{G/H}N_0$ is injective and $G$-equivariant. Let $\pi$ be the projection to the $H$-stable factor $N_0$ and write $P_0=\prod_{i\in I}P^{(i)}$. The kernel of $\pi j$ is a prime ideal, hence the product of all factors but one, $P^{(i_0)}$; it is $H$-stable, so $H$ fixes $i_0$. The stabiliser $H_0\supseteq H$ of $i_0$ acts faithfully on the field $P^{(i_0)}$, which is finite over $E$. Put $M=\pi j(P^{(i_0)})\cong P^{(i_0)}$; it is finitely generated over $k$, being finite over $E$, which is finitely generated over $k$ as a subfield of transcendence degree $\le1$ of the finitely generated field $F'$. If $\operatorname{trdeg}_kM=0$ then $M=k$ and $H$ acts trivially, a contradiction.
\end{proof}

\begin{lemma}[Abel--Jacobi]\label{lem:AJ}
Let $k=\bar k$, $L=k(x_1,\dots,x_n)$ and $N_0\supseteq L$ finite separable of degree $e$ (automatic in characteristic $0$). Let a finite group $G$ act on $N_0$ by $k$-automorphisms preserving $L$ and a subfield $M\subseteq N_0$ with $k\subseteq M$, $\operatorname{trdeg}_kM=1$, $G$ acting faithfully on $M$. Let $C$ be the smooth projective curve with $k(C)=M$, let $h_g\in\Aut_k(C)$ be given by $h_g^*=g|_M$, and $G_C=\{h_g\}$ (the map $g\mapsto h_g$ is an anti-homomorphism, so $G_C$ is a group isomorphic to $G$). Then there is a divisor $D$ on $C$, defined over $k$, of degree $e$, such that \textup{(a)} $r:=\dim|D|\ge1$ and \textup{(b)} $h(D)\sim D$ for all $h\in G_C$.
\end{lemma}
\begin{proof}
Let $\Omega$ be an algebraic closure of $N_0$ (hence of $L$), $\iota:M\hookrightarrow N_0$, and $\varepsilon_1,\dots,\varepsilon_e:N_0\to\Omega$ the $L$-embeddings. A $k$-embedding $\varphi:M\to\Omega$ is an $\Omega$-point $x_\varphi$ of $C$ lying over the generic point. Put $\mathcal D=\sum_ix_{\varepsilon_i\iota}$, an effective divisor of degree $e$ on $C_\Omega$.

The class is constant. Let $J=\Pic^0_C$, $P_0\in C(k)$, $\alpha:C\to J$, $x\mapsto[x-P_0]$, and $\beta=\sum_i\alpha(x_{\varepsilon_i\iota})\in J(\Omega)$. Since $N_0/L$ is separable, the $\varepsilon_i$ are the $e$ distinct $L$-embeddings, $\Aut(\Omega/L)$ permutes them, and $\mathcal D$ is the base change to $\Omega$ of an $L$-rational effective divisor on $C_L$ (the image of the closed point $\Spec N_0\to C_L$, with multiplicity $[N_0:L\cdot M]$); so $\beta\in J(L)$. As $L$ is purely transcendental over $k$ (and $J=0$ if $g(C)=0$), $J(L)=J(k)$ by Lemma~\ref{lem:constant}(ii). Hence $\beta=c\in J(k)$, and $[\mathcal D]=c+eP_0$ is the class of a $k$-rational divisor $D$; thus $\mathcal D\in|D|(\Omega)$.

(a). Since $N_0/L$ is separable, $\mathcal D$ is the base change to $\Omega$ of the effective $L$-rational divisor on $C_L$ obtained from the closed point $\Spec N_0\to C_L$ (with multiplicity $[N_0:L\cdot M]$), so it is an $L$-point of $|D|\cong\PP^r_k$. If $r=0$, $|D|$ is a single divisor $D_0$ over $k$, whose support lies over closed points of $C$; but each $x_{\varepsilon_i\iota}$ lies over the generic point. So $r\ge1$.

(b). Let $g\in G$ and choose $\gamma\in\Aut(\Omega/k)$ with $\gamma|_L=g|_L$. Then $\varepsilon_i'=\gamma\varepsilon_ig^{-1}$ is an $L$-embedding, $i\mapsto i'$ is a bijection, and $\gamma\varepsilon_i\iota=\varepsilon'_i\iota\,g|_M$. On generic-type points $\gamma(x_\varphi)=x_{\gamma\varphi}$ and $h_g(x_\varphi)=x_{\varphi\circ g|_M}$, so $\gamma(\mathcal D)=h_g(\mathcal D)$. Since $\gamma$ fixes $k$ and $D$ and preserves linear equivalence, $h_g(\mathcal D)=\gamma(\mathcal D)\sim D$; also $h_g(\mathcal D)\sim h_g(D)$. Hence $h_g(D)\sim D$.
\end{proof}

\begin{proposition}[the witness curve]\label{prop:witness}
Let $k=\bar k$ and let $L/K$ be a $G$-Galois extension, $G$ finite, with $L=k(x_1,\dots,x_n)$ purely transcendental over $k$. Let $F'/K$ be a finite separable extension with $\ed_k(\tau\otimes_KF')\le1$, $\tau=\Spec L$, and let $N_0,H$ be as in Section~\ref{ss:tors}, $H\ne1$; here $N_0/L$ is separable, being a factor of the \'etale algebra $L\otimes_KF'$ over $L$. Then there are a curve $C$ with a faithful action of $H$ and an effective divisor $D$ on $C$ of degree $e=[N_0:L]$ with $\ell(D)\ge2$ and $h^*D\sim D$ for every $h\in H$.
\end{proposition}
\begin{proof}
Lemma~\ref{lem:compression} gives $M\subseteq N_0$, stable under $H$, with $H$ acting faithfully and $\operatorname{trdeg}_kM=1$. Now $H$ acts on $N_0\supseteq L$, preserving $L$ and $M$, and $L$ is purely transcendental over $k$, so Lemma~\ref{lem:AJ} applies to $H$ and gives a class $D$; replacing $D$ by any $k$-point of $|D|$ (a non-empty set, as $k$ is algebraically closed and $r\ge1$) we may take $D$ effective.
\end{proof}

\begin{remark}
Proposition~\ref{prop:witness} is the first half of the lower bound for $S_5$ (Proposition~\ref{prop:S5lower}); the second half is a statement about $H$-invariant linear systems on a faithful $H$-curve, which depends on $H$.
\end{remark}

% ---- end const.tex ----

% ---- begin thmA.tex ----
\section{Valuation-theoretic lower bounds: the rank--index theorem}\label{sec:thmA}

For an abelian subgroup $A\le G$ and $d\ge1$ put
\[
\rho_A(d):=\min\{\rk B:\ B\le A,\ [A:B]\le d\}.
\]
The main result of this section, Theorem~\ref{thm:rankindex}, says that $\edd d_k(\tgen)\ge\rho_A(d)$ whenever $\operatorname{char}k\nmid|A|$: an extension of degree $d$ lowers the essential dimension of $\tgen$ at most to the rank of a subgroup of $A$ of index $\le d$. The mechanism is valuation-theoretic. A flag valuation on $L$ whose decomposition group contains $A$ (Lemma~\ref{lem:flag}) has, over every finite extension $F'$ of $K$, a decomposition group meeting $A$ in a subgroup of index $\le[F':K]_s$; and over a field of transcendence degree $j$ with algebraically closed residue field, the abelian subgroups of order prime to the residue characteristic in the image of a decomposition group have rank $\le j$ (Lemmas~\ref{lem:tame} and~\ref{lem:abhyankar}). Together with the subgroup bound and Kummer theory this determines the profile of every abelian group of exponent $e$ over a field containing $\mu_e$ (Corollary~\ref{cor:A7}). We use the conventions of Section~\ref{ss:galois}.

\subsection{Preliminaries}\label{sec:prelim}

\begin{lemma}[index of $\Gam_{F'}$]\label{lem:index}
For $K\subseteq F'\subseteq\Om$ with $[F':K]$ finite, $[\Gam_K:\Gam_{F'}]=[F':K]_s\le[F':K]$.
\end{lemma}
\begin{proof}
$\Gam_K$ acts transitively on the $K$-embeddings $F'\to\Om$, since each extends to an automorphism of $\Om$. The stabiliser of the inclusion is $\Gam_{F'}$, and there are $[F':K]_s$ embeddings.
\end{proof}

\begin{lemma}[compositum degree]\label{lem:compositum}
If $K\subseteq F'$ and $K\subseteq K_1$ inside $\Om$, with $F'/K$ finite, then $[F'K_1:K_1]\le[F':K]$.
\end{lemma}
\begin{proof}
$F'K_1=K_1[F']$ because $F'/K$ is algebraic, and it is spanned over $K_1$ by a $K$-basis of $F'$.
\end{proof}

\begin{lemma}[a valuation with algebraically closed residue field and $A$ in the decomposition group]\label{lem:flag}
Assume $k$ is algebraically closed, let $\Om$ be an algebraic closure of $L$ (so $\Om$ is algebraic over $K$ and contains, after an embedding over $K$, every finite extension $F'$ of $K$), and let $A\le G$ be abelian with $\operatorname{char}k\nmid|A|$. Then there is a valuation $\nu$ on $\Om$, trivial on $k$, such that
\begin{enumerate}[label=\textup{(\roman*)}]
\item $\kappa(M)=k$ for every field $k\subseteq M\subseteq\Om$;
\item $A\subseteq\varphi(\Dec_K)$.
\end{enumerate}
\end{lemma}
\begin{proof}
Coordinates. $A$ is finite abelian of order prime to $\operatorname{char}k$ and $k$ is algebraically closed, so $V^*$ has a basis $x_1,\dots,x_n$ of $A$-eigenvectors: $a\,x_i=\chi_i(a)x_i$ with $\chi_i(a)\in k^\times$.

The flag valuation. For $0\ne f=\sum c_mx^m\in k[x]$, let $\nu_0(f)\in\Z^n$ be the lexicographically smallest exponent $m$ with $c_m\ne0$, and let $\inn(f):=c_{\nu_0(f)}$. The lexicographic order is a total order compatible with addition, so $\nu_0(fg)=\nu_0(f)+\nu_0(g)$, $\inn(fg)=\inn(f)\inn(g)$, and $\nu_0(f+g)\ge\min(\nu_0 f,\nu_0 g)$. Hence $\nu_0$ extends to a valuation of $L=k(x)$ with value group $\Z^n$ (ordered lexicographically), trivial on $k$.

Residue field of $\nu_0$. If $\nu_0(f)=\nu_0(g)$, then
\[
\frac fg-\frac{\inn(f)}{\inn(g)}=\frac{\inn(g)f-\inn(f)g}{\inn(g)\,g},
\]
and the numerator either vanishes or has value $>\nu_0(g)$. So the residue of $f/g$ is $\inn(f)/\inn(g)\in k$, and the residue field of $\nu_0$ is $k$.

$A$ preserves $\nu_0$. $a\cdot f=\sum c_m\chi(a)^mx^m$, with $\chi(a)^m=\prod_i\chi_i(a)^{m_i}$, has the same support as $f$, so $\nu_0\circ a=\nu_0$.

Extension to $\Om$. Let $\nu$ be any extension of $\nu_0$ to $\Om$ \cite[\S3.1]{EP}.

Proof of (i). $\kappa(K)\subseteq\kappa(L)=k$ and $\kappa(K)\supseteq k$, so $\kappa(K)=k$. The residue field of $\Om$ is algebraic over $\kappa(K)$ \cite[Cor.~3.2.3]{EP}, hence equals $k$, as $k$ is algebraically closed. For any $k\subseteq M\subseteq\Om$ this gives $k\subseteq\kappa(M)\subseteq k$.

Proof of (ii). Let $a\in A$ and extend $a$ to some $\gamma\in\Gam_K$. Both $\nu\circ\gamma$ and $\nu$ extend $\nu_0\circ a=\nu_0$ from $L$ to the normal extension $\Om/L$. By the conjugation theorem \cite[Thm.~3.2.15]{EP}
there is $\delta\in\Gam_L$ with $\nu\circ\gamma=\nu\circ\delta$. Then $\gamma\delta^{-1}\in \Dec_K$ and $\varphi(\gamma\delta^{-1})=(\gamma\delta^{-1})|_L=a$.
\end{proof}

Because every $\kappa(M)$ is algebraically closed, inertia groups equal decomposition groups for this valuation; this is why only decomposition groups occur below.

\begin{lemma}[tame structure]\label{lem:tame}
Let $E\subseteq\Om$ and let $\nu$ be a valuation on $\Om$ such that $\kappa(E)$ is algebraically closed and $r:=\dim_\Q(\nu(E^\times)\otimes\Q)$ is finite. Let $\varepsilon$ be the characteristic exponent of $\kappa(E)$ ($\varepsilon=1$ if $\operatorname{char}\kappa(E)=0$, $\varepsilon=\operatorname{char}\kappa(E)$ otherwise), let $\bDec_E:=\{\sigma\in\Gal(E^s/E):\nu\circ\sigma=\nu\text{ on }E^s\}$, and let $\bar\psi\colon\Gal(E^s/E)\to G$ be a continuous homomorphism. Then every abelian subgroup $B\le\bar\psi(\bDec_E)$ of order prime to $\varepsilon$ satisfies $\rk B\le r$. The same holds for $B\le\psi(\Dec_E)$ whenever $\psi=\bar\psi\circ\rho_E$ (in particular for the classifying homomorphism of a torsor over $E$), because $\rho_E(\Dec_E)\subseteq\bDec_E$.
\end{lemma}
\begin{proof}
Step 1: reduction to the henselisation. If $\nu|_E$ is trivial then $r=0$, $E=\kappa(E)$ is algebraically closed, $E^s=E$ and there is nothing to prove. Otherwise let $Z\subseteq E^s$ be the fixed field of $\bDec_E$. Then $(Z,\nu|_Z)$ is a henselisation of $(E,\nu|_E)$ \cite[\S5.2]{EP}; it is an immediate extension, so $\nu(Z^\times)=\nu(E^\times)=:\Delta_E$ and $\kappa(Z)=\kappa(E)$ \cite[\S5.2]{EP}; and $\bDec_E=\Gal(E^s/Z)$. (If $\sigma\in\Dec_E$ then $\nu|_{E^s}\circ\sigma|_{E^s}=\nu|_{E^s}$, so $\rho_E(\Dec_E)\subseteq\bDec_E$.)

Step 2: inertia and ramification. $\kappa(E^s)$ is algebraic over the algebraically closed field $\kappa(Z)=\kappa(E)$, so the two are equal, and the inertia group of $\nu|_{E^s}$ over $Z$ is all of $\bDec_E$ (inertia equals decomposition because the residue field is algebraically closed; this is what lets us apply the description of inertia modulo ramification below to $\bDec_E$). Let $R\trianglelefteq \bDec_E$ be the ramification group; it is a pro-$\varepsilon$ group (trivial if $\varepsilon=1$), and, with $\Delta:=\nu((E^s)^\times)$,
\[
\bDec_E/R\cong\Hom(\Delta/\Delta_E,\kappa(E)^\times)
\]
as profinite groups, the standard description of tame inertia over a henselian field \cite[Thm.~5.3.3]{EP}.
Moreover $\Delta/\Delta_E$ is torsion because $E^s/E$ is algebraic \cite[Cor.~3.2.3]{EP}, i.e.\ $\Delta\subseteq\Delta_E\otimes\Q$. Hence every homomorphism $\Delta/\Delta_E\to\kappa(E)^\times$ takes values in the group $\mu_\infty(\kappa(E))$ of roots of unity of $\kappa(E)$, and $\Hom(\Delta/\Delta_E,\kappa(E)^\times)=\Hom(\Delta/\Delta_E,\mu_\infty(\kappa(E)))$.

Step 3: removing the $\varepsilon$-part. Put $P:=\bar\psi(R)$ and $Q:=\bar\psi(\bDec_E)$. As $\bar\psi$ is continuous and $G$ is finite, $P$ is a finite quotient of the pro-$\varepsilon$ group $R$, hence an $\varepsilon$-group, and it is normal in $Q$ because $R\trianglelefteq \bDec_E$. Since $|B|$ is prime to $\varepsilon$, $B\cap P=1$, so $B$ embeds in $Q/P$, and $\bar\psi$ induces a continuous surjection $\bDec_E/R\to Q/P$.

Step 4: finite continuous quotients of $\bDec_E/R$ have rank $\le r$. Let $X:=\Delta/\Delta_E$, a discrete torsion group, and $X'$ its prime-to-$\varepsilon$ part. As $\kappa(E)$ is algebraically closed of characteristic exponent $\varepsilon$, $\mu_\infty(\kappa(E))$ is isomorphic to the prime-to-$\varepsilon$ part of $\Q/\Z$; so $\bDec_E/R\cong\Hom(X',\mu_\infty(\kappa(E)))$ is the Pontryagin dual of $X'$. A finite continuous quotient of the dual of the discrete torsion group $X'$ is the dual $\hat S$ of a finite subgroup $S\subseteq X'$ (an open subgroup $U$ of the dual is the annihilator of the finite subgroup $S=U^\perp$, and the quotient by $U$ is $\hat S$), and $\hat S\cong S$. For each prime $\ell$, inside $\Delta_E\otimes\Q$,
\[
S[\ell]\subseteq X[\ell]\subseteq(\ell^{-1}\Delta_E)/\Delta_E\cong\Delta_E/\ell\Delta_E,
\]
the last isomorphism being multiplication by $\ell$ ($\Delta_E$ is torsion-free). Finally $\dim_{\F_\ell}\Delta_E/\ell\Delta_E\le r$: given $r+1$ elements $\gamma_i$ of $\Delta_E$ there is a relation $\sum n_i\gamma_i=0$ with integers $n_i$ not all zero, and dividing by their greatest common divisor (using that $\Delta_E$ is torsion-free) we may assume $\gcd(n_i)=1$, so some $n_i\not\equiv0\pmod\ell$. Hence $\rk S=\max_\ell\dim_{\F_\ell}S[\ell]\le r$.

Step 5. $Q/P$ is a finite continuous quotient of $\bDec_E/R$, hence abelian of rank $\le r$, and $B$ is isomorphic to a subgroup of $Q/P$; so $\rk B\le r$.
\end{proof}

\begin{lemma}[Abhyankar's inequality]\label{lem:abhyankar}
Let $E/k$ be a field extension of finite transcendence degree, not necessarily finitely generated, and $w$ a valuation on $E$ trivial on $k$. Then
\[
\dim_\Q(w(E^\times)\otimes\Q)+\trdeg_k\kappa(w)\le\trdeg_kE.
\]
\end{lemma}
\begin{proof}
For $E$ finitely generated over $k$ this is \cite[VI~\S10.3, Cor.~1]{Bou}, \cite[Ch.~3]{EP}.
In general, let $x_1,\dots,x_s\in E^\times$ have $\Q$-linearly independent values and let $y_1,\dots,y_t$ be elements of the valuation ring whose residues are algebraically independent over $k$. The finitely generated field $E_0=k(x,y)\subseteq E$ has, for $w|_{E_0}$, rational rank $\ge s$ and residue transcendence degree $\ge t$, so $s+t\le\trdeg_kE_0\le\trdeg_kE$. Taking $s$ and $t$ as large as possible gives the claim.
\end{proof}

\subsection{The rank--index theorem}

\begin{theorem}[Rank--index bound]\label{thm:rankindex}
Let $k$ be any field, $G$ a finite group, $V$ a faithful representation of $G$ over $k$ with generic torsor $\tgen$ over $K$, and $A\le G$ an abelian subgroup with $\operatorname{char}k\nmid|A|$.
\begin{enumerate}[label=\textup{(\alph*)}]
\item For every finite field extension $F'/K$, $\ed_k(\tgen\otimes_KF')\ge\rho_A([F':K]_s)$.
\item For every $d\ge1$, $\edd d_k(\tgen)\ge\rho_A(d)$. Equivalently, $d_j(\tgen)\ge\min\{[A:B]:\ B\le A,\ \rk B\le j\}$ for $0\le j\le\ed_k(\tgen)$.
\end{enumerate}
\end{theorem}

The engine is the following proposition, which Theorem~\ref{thm:Zp} also uses.

\begin{proposition}\label{prop:A4}
Assume $k$ is algebraically closed; let $A$ and $\nu$ be as in Lemma~\ref{lem:flag}. Let $F'\supseteq K$ be a field inside $\Om$, finite over $K$ with $[F':K]_s\le d$, and $k\subseteq E\subseteq F'$ a subfield with $\tgen\otimes_KF'\cong\tau_0\otimes_EF'$ for a $G$-torsor $\tau_0$ over $E$ classified by $\psi\colon\Gam_E\to G$. Then there is $h\in G$ such that $B':=A\cap\varphi(\Dec_{F'})$ satisfies
\begin{enumerate}[label=\textup{(\alph*)}]
\item $[A:B']\le d$;
\item $h^{-1}B'h\subseteq\psi(\Dec_E)$.
\end{enumerate}
\end{proposition}
\begin{proof}
(a) By Lemma~\ref{lem:index}, $[\Gam_K:\Gam_{F'}]=[F':K]_s\le d$. Since $\Dec_{F'}=\Dec_K\cap\Gam_{F'}$, the coset map $\Dec_K/\Dec_{F'}\to\Gam_K/\Gam_{F'}$ is injective, so $[\Dec_K:\Dec_{F'}]\le d$. Applying $\varphi$, the groups $I:=\varphi(\Dec_K)$ and $I':=\varphi(\Dec_{F'})$ satisfy $[I:I']\le[\Dec_K:\Dec_{F'}]\le d$. Now $A\subseteq I$ by Lemma~\ref{lem:flag}(ii), and $A/(A\cap I')\hookrightarrow I/I'$ as coset sets. So $[A:B']\le d$.

(b) The torsor isomorphism means that $\varphi|_{\Gam_{F'}}$ and $\psi|_{\Gam_{F'}}$ are conjugate (Section~\ref{ss:galois}; note $\Gam_{F'}\subseteq\Gam_E$ since $E\subseteq F'$): there is $h\in G$ with $\varphi(\gamma)=h\psi(\gamma)h^{-1}$ for all $\gamma\in\Gam_{F'}$. As both decomposition groups are taken for the same $\nu$, $\Dec_{F'}=\Dec_E\cap\Gam_{F'}\subseteq \Dec_E$. Hence
\[
B'\subseteq\varphi(\Dec_{F'})=h\,\psi(\Dec_{F'})\,h^{-1}\subseteq h\,\psi(\Dec_E)\,h^{-1}. \qedhere
\]
\end{proof}

\begin{proof}[Proof of Theorem~\ref{thm:rankindex}]
Step 1: reduction to $k$ algebraically closed. Let $F'/K$ be finite and suppose $\ed_k(\tgen\otimes F')=j$, realised by $k\subseteq E\subseteq F'$ with $\trdeg_kE=j$ and $\tgen\otimes_KF'\cong\tau_0\otimes_EF'$. Put everything in an algebraic closure $\Om$ of $F'$ (with $L$ embedded over $K$), let $\kb\subseteq\Om$ be the algebraic closure of $k$, and $\tilde K=K\kb$, $\tilde F=F'\kb$, $\tilde E=E\kb$. Let $F'_s$ be the separable closure of $K$ in $F'$. Then $F'_s\kb$ is separable over $\tilde K$ with $[F'_s\kb:\tilde K]\le[F'_s:K]$ (Lemma~\ref{lem:compositum}; no linear disjointness is needed), and $\tilde F$ is generated over $F'_s\kb$ by elements purely inseparable over $F'_s$; so $[\tilde F:\tilde K]_s=[F'_s\kb:\tilde K]\le[F':K]_s$. Also $\trdeg_{\kb}\tilde E=\trdeg_kE=j$. By Lemma~\ref{lem:regular}(c), $\tgen\otimes_K\tilde K$ is the generic torsor of $V_{\kb}$ over $\kb$, and
\[
(\tgen\otimes_K\tilde K)\otimes_{\tilde K}\tilde F=\tgen\otimes_K\tilde F\cong(\tau_0\otimes_E\tilde E)\otimes_{\tilde E}\tilde F,
\]
so $\ed_{\kb}((\tgen\otimes_K\tilde K)\otimes_{\tilde K}\tilde F)\le j$. The hypothesis $\operatorname{char}\kb\nmid|A|$ is unchanged and $\rho_A$ is non-increasing, so (a) over $\kb$ gives $j\ge\rho_A([\tilde F:\tilde K]_s)\ge\rho_A([F':K]_s)$. It therefore suffices to prove (a) for $k$ algebraically closed.

Step 2: the case $k$ algebraically closed. Let $\nu$ be the valuation of Lemma~\ref{lem:flag}, embed $F'$ in $\Om$ over $K$ (possible since $F'/K$ is finite and $\Om$ is algebraically closed), and let $E$, $\tau_0$, $\psi$ be as in Proposition~\ref{prop:A4}, with $d=[F':K]_s$ and $\trdeg_kE=j$; $\psi=\bar\psi\circ\rho_E$ with $\bar\psi$ continuous, as $\tau_0$ is a torsor. Let $B=h^{-1}B'h$. By Proposition~\ref{prop:A4}, $B\le h^{-1}Ah$ has index $\le d$ and $B\subseteq\psi(\Dec_E)$. By Lemma~\ref{lem:flag}(i), $\kappa(E)=k$ is algebraically closed, of characteristic exponent prime to $|A|$. Lemma~\ref{lem:abhyankar} applied to $\nu|_E$, which is trivial on $k$, gives
\[
r:=\dim_\Q(\nu(E^\times)\otimes\Q)\le j-\trdeg_k\kappa(E)=j
\]
($r=0$ if $\nu|_E$ is trivial). By Lemma~\ref{lem:tame} (the last sentence: $\psi$ factors through $\rho_E$, since the total space of $\tau_0$ is finite \'etale over $E$), $\rk B\le r\le j$, since $B$ is abelian of order dividing $|A|$. Therefore $B'=hBh^{-1}\le A$ has $[A:B']\le d$ and $\rk B'\le j$, so $j\ge\rho_A(d)$. This proves (a).

Step 3: (b). If $[F':K]\le d$ then $[F':K]_s\le d$ and $\rho_A([F':K]_s)\ge\rho_A(d)$, so (a) gives the profile bound. If $\edd d_k(\tgen)\ge\rho_A(d)$ for all $d$ and $d=d_j(\tgen)$, then $\rho_A(d)\le\edd d_k(\tgen)\le j$, so some $B\le A$ has $[A:B]\le d$ and $\rk B\le j$. Conversely, if the $d_j$-bound holds for all $j$, then for $j:=\edd d_k(\tgen)$ Proposition~\ref{prop:shape}(c) gives $d_j(\tgen)\le d$, so some $B\le A$ with $\rk B\le j$ has $[A:B]\le d$, i.e.\ $\rho_A(d)\le\edd d_k(\tgen)$.
\end{proof}

\begin{remark}\label{rem:RY}
(i) For $d=1$ the theorem gives $\ed_k(G)\ge\rk A$ for every abelian $A\le G$ with $\operatorname{char}k\nmid|A|$. This lower bound is classical: for finite groups it follows from \cite[Thm.~6.1 and Lemma~4.1(a)]{BR97} when $k$ contains enough roots of unity, and for every field with $\operatorname{char}k\nmid|A|$ from \cite[Prop.~3.7, Rem.~3.8 and Thm.~6.19]{BF03}; the algebraic-group version over an algebraically closed field of characteristic $0$ is \cite[Cor.~7.3, Ex.~7.4]{RY}. No novelty is claimed for $d=1$. What the argument adds is the control of the index, i.e.\ the dependence on $d$.
(ii) No hypothesis on $k$ beyond $\operatorname{char}k\nmid|A|$ is needed, and $V$ is an arbitrary faithful representation. The reduction to $\kb$ uses only the inequality of Lemma~\ref{lem:compositum}, not linear disjointness of $F'$ and $\kb$; after the reduction all residue fields are $\kb$ (Lemma~\ref{lem:flag}(i)), so the question of a non-algebraically-closed residue field of $E$ does not arise.
(iii) Part (a) is a statement about each extension, and it involves only the separable degree. In characteristic $\ell>0$, $\ell\nmid|A|$, it shows for instance that a purely inseparable extension $F'/K$ never lowers $\ed_k(\tgen\otimes F')$ below $\rk A$, over any $k$; for perfect $k$, Lemma~\ref{lem:pi} gives the stronger equality $\ed_k(\tgen\otimes F')=\ed_k(\tgen)$.
\end{remark}

\subsection{Abelian groups}\label{ss:abelian}
The subgroup bound (Corollary~\ref{cor:sub}) supplies the matching upper bound once the subgroups $B$ attaining $\rho_A(d)$ have essential dimension $\le\rk B$.

\begin{lemma}[Kummer]\label{lem:kummer}
Let $B$ be a finite abelian group of exponent $e$, $\operatorname{char}k\nmid e$, $\mu_e\subset k$. Then $\ed_k(B)\le\rk B$.
\end{lemma}
\begin{proof}
By the invariant-factor decomposition, $B\cong\bigoplus_{i=1}^r\Z/e_i$ with $1<e_1\mid e_2\mid\dots\mid e_r\mid e$; here $r=\rk B$, since $B$ is generated by $r$ elements and $B/\ell B\cong\F_\ell^r$ for a prime $\ell\mid e_1$. Fix isomorphisms $\Z/e_i\cong\mu_{e_i}\subset k^\times$. For a field $F\supseteq k$, Kummer theory gives $H^1(F,\Z/e_i)\cong H^1(F,\mu_{e_i})\cong F^\times/F^{\times e_i}$, and $H^1(F,B)=\prod_iH^1(F,\Z/e_i)$. So every $B$-torsor over $F$ is
\[
\textstyle\bigotimes_{F,i}F[t_i]/(t_i^{e_i}-a_i)\qquad(a_i\in F^\times),
\]
the tensor product over $F$ of the Kummer algebras (of dimension $\prod_ie_i=|B|$), with $\Z/e_i$ acting on $t_i$ through $\mu_{e_i}$. It is the base change of the $B$-torsor $\bigotimes_{E_0,i}E_0[t_i]/(t_i^{e_i}-a_i)$ over $E_0:=k(a_1,\dots,a_r)$, and $\trdeg_kE_0\le r$.
\end{proof}
With Theorem~\ref{thm:rankindex} for $d=1$ and $A=G=B$ this gives $\ed_k(B)=\rk B$ under the hypotheses of the lemma; in characteristic $0$ this equality is due to Buhler and Reichstein \cite[Thm.~6.1]{BR97}.

\begin{corollary}[profile of abelian groups]\label{cor:A7}
Let $G$ be abelian of exponent $e$, with $\operatorname{char}k\nmid e$ and $\mu_e\subset k$. For every faithful $V$,
\[
\edd d_k(\tgen)=\min\{\rk B:B\le G,\ [G:B]\le d\},\qquad d_j(\tgen)=\min\{[G:B]:B\le G,\ \rk B\le j\}.
\]
\end{corollary}
\begin{proof}
``$\ge$'' is Theorem~\ref{thm:rankindex} with $A=G$. For ``$\le$'', choose $B$ attaining the minimum; then $\edd d_k(\tgen)\le\edd{[G:B]}_k(\tgen)\le\ed_k(B)\le\rk B$ by monotonicity, Corollary~\ref{cor:sub} and Lemma~\ref{lem:kummer} ($B$ has exponent dividing $e$). The formula for $d_j$ follows from the definition.
\end{proof}

\begin{corollary}[elementary abelian groups]\label{cor:A8}
Let $G=(\Z/q)^n$, $q$ prime, $\operatorname{char}k\ne q$.
\begin{enumerate}[label=\textup{(\alph*)}]
\item For every such $k$: $\edd d_k(\tgen)\ge n-\lfloor\log_qd\rfloor$ for $1\le d\le q^n$, and $d_j(\tgen)\ge q^{n-j}$ for $0\le j\le n$.
\item If $\ed_k(\Z/q)\le1$, then $d_j(\tgen)=q^{n-j}$ for $0\le j\le n$, and the full profile is $\edd d_k(\tgen)=n-\lfloor\log_qd\rfloor$ for $1\le d\le q^n$ and $0$ for $d\ge q^n$; in particular $\ed_k((\Z/q)^n)=n$. The hypothesis $\ed_k(\Z/q)\le1$ holds if $\mu_q\subset k$ (Lemma~\ref{lem:kummer}), hence always for $q=2$, and, for $q$ odd, if $\zeta_q+\zeta_q^{-1}\in k$ (Lemma~\ref{lem:B1}).
\item In particular $d_j((\Z/2)^n)=2^{n-j}$ for every $j$, over every field of characteristic $\ne2$, algebraically closed or not.
\end{enumerate}
\end{corollary}
\begin{proof}
A subgroup $B\le(\Z/q)^n$ is elementary abelian, so $|B|=q^{\rk B}$. Hence $\rk B\le j$ implies $[G:B]\ge q^{n-j}$, with equality for $B\cong(\Z/q)^j$; and for $1\le d\le q^n$, $[G:B]\le d$ implies $|B|\ge q^n/d$, so $\rk B\ge\lceil n-\log_qd\rceil=n-\lfloor\log_qd\rfloor$, with equality for $B\cong(\Z/q)^{n-\lfloor\log_qd\rfloor}$, whose index $q^{\lfloor\log_qd\rfloor}$ is $\le d$. Thus $\rho_G(d)=n-\lfloor\log_qd\rfloor$, and (a) is Theorem~\ref{thm:rankindex} with $A=G$. For (b), let $B\cong(\Z/q)^m\le G$. A $B$-torsor over $F\supseteq k$ is an $m$-tuple of $\Z/q$-torsors ($H^1(F,B)=H^1(F,\Z/q)^m$); each descends to a subfield of transcendence degree $\le\ed_k(\Z/q)\le1$, and all of them to the compositum of these subfields, so $\ed_k(B)\le m$. By Corollary~\ref{cor:sub}, $\edd{q^{n-m}}_k(\tgen)\le m$; with $m=n-\lfloor\log_qd\rfloor$ and monotonicity this is the upper bound $\edd d_k(\tgen)\le n-\lfloor\log_qd\rfloor$, and with $m=0$ it gives $\edd{q^n}_k(\tgen)=0$. Together with (a) this proves (b), and (c) is the case $q=2$.
\end{proof}

\begin{remark}\label{rem:A8}
(i) Corollary~\ref{cor:A8} holds over every field of characteristic $\ne2$. Corollary~\ref{cor:A8} also improves the case $G=(\Z/q)^n$ of Proposition~\ref{prop:d0}(c), which needs $k$ perfect: (a) gives $\edd d_k(\tgen)\ge1$ for all $d<q^n$ over every $k$ with $\operatorname{char}k\ne q$ (if $[G:B]\le d<q^n$ then $B\ne1$). There is no cohomological input: the lower bound is valuation-theoretic and the upper bound is Kummer theory, or Lemma~\ref{lem:B1}.
(ii) If $q$ is odd and $\zeta_q+\zeta_q^{-1}\notin k$, put $k'=k(\zeta_q+\zeta_q^{-1})$. By Proposition~\ref{prop:shape}(b) with $F'=Kk'$, Lemma~\ref{lem:regular}(c), Proposition~\ref{prop:shape}(d) and (b) over $k'$,
\[
q^{n-j}\ \le\ d_{j,k}(\tgen)\ \le\ [k':k]\cdot d_{j,k}(\tgen\otimes Kk')=[k':k]\cdot d_{j,k'}(\tgen(V_{k'}))=[k':k]\,q^{n-j}.
\]
For $n=1$ the left side is trivial and Theorem~\ref{thm:Zp} gives the sharper lower bound for $d_1$; for $n\ge2$ the true value is not known in general (the smallest case is $(\Z/5)^2$ over $\Q$, where $5\mid d_1$ by Theorem~\ref{thm:div}(b), so $d_{1,\Q}(\tgen)\in\{5,10\}$).
\end{remark}

\begin{corollary}[the subgroup formula for abelian groups]\label{cor:subgroupabelian}
Let $G$ be abelian of exponent $e$, $\operatorname{char}k\nmid e$, $\mu_e\subset k$. Then $\rdim_k(H)=\rk H$ for every $H\le G$, and
\[
\edd d_k(\tgen)=\min\{\rdim_k(H):H\le G,\ [G:H]\le d\}\qquad\text{for all }d\ge1.
\]
In particular equality holds in \eqref{eq:subgroupformula} for every abelian $p$-group $G$ of exponent $p^a$ when $\mu_{p^a}\subset k$.
\end{corollary}
\begin{proof}
Let $H\le G$. A representation of $H$ over $k$ is a direct sum of characters $H\to\mu_e\subset k^\times$, since the operators commute and are annihilated by $X^e-1$, which splits over $k$ with distinct roots. It is faithful iff the intersection of the kernels of its characters is trivial, iff these characters generate the character group $\hat H=\Hom(H,\mu_e)$ (by duality of finite abelian groups, a subgroup of $\hat H$ whose elements have trivial common kernel is all of $\hat H$). As $\mu_e$ is cyclic of order $e$ and the exponent of $H$ divides $e$, $\hat H\cong H$, so $\rdim_k(H)$ is the minimal number of generators of $\hat H$, namely $\rk H$. Now apply Corollary~\ref{cor:A7}.
\end{proof}

\begin{remark}\label{rem:Z4Q}
The hypothesis $\mu_e\subset k$ in Corollaries~\ref{cor:A7} and~\ref{cor:subgroupabelian} cannot be weakened to $\zeta_p\in k$, the hypothesis of the subgroup formula: for $G=\Z/4$ and $k=\Q$ (where $\zeta_2=-1\in k$) one has $\rk G=1$ but $\ed_\Q(\Z/4)=\rdim_\Q(\Z/4)=[\Q(i):\Q]=2$ (Theorem~\ref{thm:KM}, Example~\ref{ex:rdim}), so the rank formula of Corollary~\ref{cor:A7} fails at $d=1$, and Theorem~\ref{thm:rankindex} is not sharp there. The subgroup formula itself does hold for $\Z/4$ over $\Q$: the right side of \eqref{eq:subgroupformula} is $2,1,1,0$ for $d=1$, $d\in\{2,3\}$, $d\ge4$; and $\edd1_\Q(\tgen)=\ed_\Q(\Z/4)=2$, $\edd d_\Q(\tgen)\le\ed_\Q(\Z/2)=1$ for $d\in\{2,3\}$ (Corollary~\ref{cor:sub}) with equality by Proposition~\ref{prop:d0}(c), and $\edd d_\Q(\tgen)=0$ for $d\ge4$.
\end{remark}

\subsection{Limits of the method}\label{ss:limits}

\begin{remark}\label{rem:limits}
(i) Theorem~\ref{thm:rankindex} says nothing beyond $d_0$ for cyclic groups. Every abelian subgroup $A$ of a cyclic group is cyclic, and for $d<|A|$ every $B\le A$ of index $\le d$ is non-trivial, so $\rho_A(d)=1$; the theorem gives only $\edd d_k(\tgen)\ge1$ for $d<|A|$. This is why the computation of $d_1(\tgen(\Z/p))$ in Section~\ref{sec:thmB} needs a different, curve-theoretic method, and why the first open case of the subgroup formula \eqref{eq:subgroupformula} is cyclic.
(ii) Even at $d=1$ the bound $\max_A\rho_A(1)=\max_A\rk A$ can be smaller than $\ed_k(G)$ for non-abelian $G$. For the quaternion group $Q_8$ and $k=\mathbb C$ every abelian subgroup is cyclic, so the theorem gives $\ed_{\mathbb C}(Q_8)\ge1$ only, while $\ed_{\mathbb C}(Q_8)=\rdim_{\mathbb C}(Q_8)=2$ by Theorem~\ref{thm:KM}: $Q_8$ is non-abelian, so it has no faithful one-dimensional representation, and its two-dimensional irreducible representation is faithful.
\end{remark}

% ---- end thmA.tex ----

% ---- begin thmB.tex ----
\section{Cyclic groups of prime order in characteristic 0}\label{sec:thmB}

Throughout this section $p$ is an odd prime, $G=\Z/p=\langle\sigma\rangle$, $V$ is a faithful representation of $G$ over $k$ with generic torsor $\tgen$ over $K$, $\zeta\in\kb$ is a primitive $p$-th root of unity, and
\[
c:=\zeta+\zeta^{-1},\qquad m:=[k(\zeta):k],\qquad \Theta:=\Gal(k(\zeta)/k)\subseteq(\Z/p)^\times,
\]
where $a\in(\Z/p)^\times$ corresponds to $\zeta\mapsto\zeta^a$. Lemma~\ref{lem:B1} and Corollary~\ref{cor:pertorsor} only need $\operatorname{char}k\ne p$; from Section~\ref{ss:lowerB} on, $\operatorname{char}k=0$.

\begin{lemma}\label{lem:cyclo}
Let $\operatorname{char}k\ne p$. Then $\Gal(k(\zeta)/k(c))=\Theta\cap\{\pm1\}$, and $-1\in \Theta$ iff $m$ is even. Hence: if $m$ is even, then $\zeta\notin k(c)$ and $[k(c):k]=m/2$; if $m$ is odd, then $k(c)=k(\zeta)$.
\end{lemma}
\begin{proof}
$\zeta$ is a root of $T^2-cT+1$, so $[k(\zeta):k(c)]\le2$ and $\Gal(k(\zeta)/k(c))$ consists of the $a\in \Theta$ fixing $c$. Now $\zeta^a+\zeta^{-a}=\zeta+\zeta^{-1}$ iff $T^2-(\zeta^a+\zeta^{-a})T+1=T^2-cT+1$ iff $\{\zeta^a,\zeta^{-a}\}=\{\zeta,\zeta^{-1}\}$ iff $a=\pm1$. Since $(\Z/p)^\times$ is cyclic, $-1$ is its unique element of order $2$, and the subgroup $\Theta$ of order $m$ contains it iff $m$ is even.
\end{proof}

\begin{theorem}[the first jump degree of $\Z/p$]\label{thm:Zp}
Let $\operatorname{char}k=0$. Then
\[
\Big\lceil\tfrac m2\Big\rceil\ \le\ d_{1,k}(\tgen)\ \le\ [k(c):k].
\]
If $m$ is even (equivalently $\zeta\notin k(c)$), the two bounds coincide: $d_{1,k}(\tgen)=[k(c):k]=m/2$. In particular, over $k=\Q$, $d_{1,\Q}(\tgen(\Z/p))=(p-1)/2$.
\end{theorem}

The upper bound holds for every $\Z/p$-torsor (Corollary~\ref{cor:pertorsor}); the lower bound is the new content, and its proof (Section~\ref{ss:lowerB}) is curve-theoretic, since the rank--index theorem is vacuous for cyclic groups (Remark~\ref{rem:limits}).

\subsection{The upper bound}

By Proposition~\ref{prop:bfe} with $k'=k(c)$, the upper bound follows from $\ed_{k(c)}(\Z/p)\le1$. That $\ed_k(\Z/p)=1$ iff $\zeta+\zeta^{-1}\in k$, in characteristic $0$, is classical (a special case of Ledet's embedding criterion \cite{Ledet07}; see \cite[Thm.~1.3(1)]{CHKZ});
we give a self-contained proof of the direction needed, valid whenever $\operatorname{char}k\ne p$.

\begin{lemma}\label{lem:B1}
Let $\operatorname{char}k\ne p$ and $k':=k(c)$. Every $\Z/p$-torsor over a field $F\supseteq k'$ descends to a subfield $E$ with $k'\subseteq E\subseteq F$ and $\trdeg_{k'}E\le1$. Hence $\ed_{k'}(\Z/p)\le1$.
\end{lemma}
This is the classical upper bound \cite[Prop.~7.4]{BF03} (see also \cite[Thm.~1.3]{CHKZ}); the proof below is included for completeness.
\begin{proof}
Reductions. A split torsor descends to $k'$, and if $F$ is finite then $E=F$ will do. Otherwise, as $|G|=p$ is prime, the torsor is $\Spec M$ for a cyclic field extension $M/F$ of degree $p$ with $G=\Gal(M/F)$, and $F$ is infinite.

A rank-two representation. Let $W=k'^2$ with $\sigma$ acting by
\[
S=\begin{pmatrix}0&-1\\1&c\end{pmatrix},\qquad \det(T-S)=T^2-cT+1=(T-\zeta)(T-\zeta^{-1}).
\]
Since $p$ is odd, $\zeta\ne\zeta^{-1}$, so $S$ is diagonalisable over $k'(\zeta)$ with eigenvalues $\zeta^{\pm1}$, $S^p=1$ and $S\ne1$. Let $x_1,x_2\in W^*$ be the coordinate functions; $\sigma$ acts on $W^*$ (contragrediently) by $\sigma x_i=s_{i1}x_1+s_{i2}x_2$ for a matrix $(s_{ij})\in\mathrm{GL}_2(k')$, again diagonalisable with eigenvalues $\zeta^{\pm1}$. Let $\mu(u)=(s_{11}u+s_{12})/(s_{21}u+s_{22})$ be the corresponding M\"obius transformation of $\PP^1(\kb)$. Its fixed points are the eigenlines of $(s_{ij})$ written as $[u:1]$, so $\mu$ has exactly two fixed points $\alpha,\beta\in\PP^1(k'(\zeta))$; one of them may be $\infty$.

Equivariant points. Let $U:=(W\otimes_{k'}M)^G$, with $G$ acting diagonally. By Galois descent for the $G$-Galois extension $M/F$, the natural map $U\otimes_FM\to W\otimes_{k'}M$ is an isomorphism; so $\dim_FU=2$ and $U$ spans $W\otimes_{k'}M$ over $M$. Each $w\in U$ defines a $k'$-algebra map $f_w\colon k'[W]=\operatorname{Sym}(W^*)\to M$ by $f_w(x)=\langle x,w\rangle$ for $x\in W^*$ (the $M$-linear extension of the evaluation pairing), and $f_w$ is $G$-equivariant because $w$ is $G$-invariant.

Avoiding the fixed points. Consider the $F$-linear maps $U\to M$ given by $w\mapsto f_w(x_2)$ and, for each fixed point $\gamma\in\{\alpha,\beta\}$ that is finite and lies in $M$, $w\mapsto f_w(x_1)-\gamma f_w(x_2)$. They are nonzero, because the $M$-linear forms $x_2$ and $x_1-\gamma x_2$ on $W\otimes_{k'}M$ are nonzero and $U$ spans $W\otimes_{k'}M$ over $M$. So their kernels are at most three proper $F$-subspaces of $U$, and since $F$ is infinite some $w\in U$ lies in none of them.

A small invariant subfield. Put $u=f_w(x_1)/f_w(x_2)\in M$; this is defined since $f_w(x_2)\ne0$. By equivariance,
\[
\sigma(u)=\frac{f_w(\sigma x_1)}{f_w(\sigma x_2)}=\frac{s_{11}u+s_{12}}{s_{21}u+s_{22}}=\mu(u).
\]
The point $u\in\PP^1(\kb)$ is not fixed by $\mu$: $u\ne\infty$; $u\ne\gamma$ for a finite fixed point $\gamma\in M$ by the choice of $w$; and $u\ne\gamma$ for a fixed point $\gamma\notin M$, as $u\in M$. So $\sigma(u)\ne u$. The field $M_1:=k'(u)\subseteq M$ is $G$-stable ($\mu$ has coefficients in $k'$), $G$ acts faithfully on it ($\sigma$ moves $u$ and $|G|$ is prime), and $\trdeg_{k'}M_1\le1$.

Descent. $E:=M_1^G\subseteq M^G=F$ contains $k'$, and $M_1/E$ is $G$-Galois (Artin). The $F$-algebra map $M_1\otimes_EF\to M$ is a $G$-equivariant morphism of $G$-torsors over $F$, hence an isomorphism. So the torsor descends to $E$, and $\trdeg_{k'}E=\trdeg_{k'}M_1\le1$.
\end{proof}

\begin{corollary}[the per-torsor bound]\label{cor:pertorsor}
Let $\operatorname{char}k\ne p$. For every $\Z/p$-torsor $\tau$ over every field $F\supseteq k$,
\[
\edd{[k(c):k]}_k(\tau)\le1,\qquad\text{i.e. }\ d_{1,k}(\tau)\le[k(c):k]\ \text{ if }\ed_k(\tau)\ge1.
\]
For $\tau=\tgen$ the field $Kk(c)$, of degree $[k(c):k]$ over $K$ (Lemma~\ref{lem:regular}(b)), is a witness. This is the upper bound of Theorem~\ref{thm:Zp}.
\end{corollary}
\begin{proof}
Let $F'$ be a field factor of $F\otimes_kk(c)$, so $[F':F]\le[k(c):k]$ and $F'\supseteq k(c)$. By Proposition~\ref{prop:bfe} and Lemma~\ref{lem:B1}, $\ed_k(\tau\otimes F')=\ed_{k(c)}(\tau\otimes F')\le\ed_{k(c)}(\Z/p)\le1$. For $\tgen$, $K\otimes_kk(c)=Kk(c)$ is a field by Lemma~\ref{lem:regular}(b).
\end{proof}

\subsection{The lower bound}\label{ss:lowerB}
Let $\operatorname{char}k=0$. Fix a finite field extension $F'\supseteq K$ of degree $d:=[F':K]$ and a subfield $E\subseteq F'$ containing $k$ with $\trdeg_kE\le1$ and $\tgen\otimes_KF'\cong\tau_0\otimes_EF'$ for a $G$-torsor $\tau_0$ over $E$. We show $d\ge\frac m2$. Since $\frac m2\le\frac{p-1}2<p$, we may assume $d<p$. Let $\Om$ be an algebraic closure of $F'$, embed $L$ in $\Om$ over $K$, and let $\kb\subseteq\Om$ be the algebraic closure of $k$.

\subsubsection*{Step 1: constant fields}
Let $k_0$ be the algebraic closure of $k$ in $F'$; it is finite over $k$ since $F'$ is finitely generated over $k$. Put $m_0=[k_0:k]$. By Lemma~\ref{lem:regular}(b), $[Kk_0:K]=m_0$, so $d':=[F':Kk_0]=d/m_0$. Replace $E$ by $Ek_0\subseteq F'$ and $\tau_0$ by $\tau_0\otimes_EEk_0$. Then $E\supseteq k_0$, $\trdeg_kE\le1$, and $k_0$ is algebraically closed in $E$ (as it is in $F'$). In characteristic $0$, $E/k_0$ is therefore regular.

\subsubsection*{Step 2: passing to $\kb$}
Let $\tilde K=K\kb$, $\tilde F=F'\kb$, $\tilde E=E\kb$, $\tilde L=L\kb$. Since $\tilde K=(Kk_0)\kb$, Lemma~\ref{lem:compositum} gives $d'':=[\tilde F:\tilde K]\le d'=d/m_0<p$. By Lemma~\ref{lem:regular}(b),(c), $\tilde L=\kb(V)$, $\tilde K=\kb(V)^G$ and $[\tilde L:\tilde K]=p$. The degree of $\tilde L\cap\tilde F$ over $\tilde K$ divides $p$ and is at most $d''<p$, so $\tilde L\cap\tilde F=\tilde K$. As $\tilde L/\tilde K$ is Galois, it follows that
\[
\Lambda:=\tilde L\tilde F\cong\tilde L\otimes_{\tilde K}\tilde F
\]
is a field with $[\Lambda:\tilde L]=[\tilde F:\tilde K]=d''$. By Lemma~\ref{lem:regular}(c), $\tgen\otimes_K\tilde F\cong\tilde L\otimes_{\tilde K}\tilde F=\Lambda$ as $G$-algebras; so $\tgen\otimes_K\tilde F$ is a non-trivial torsor, $G=\Gal(\Lambda/\tilde F)$, and $G$ acts on $\tilde L\subseteq\Lambda$ as on $\kb(V)$.

\subsubsection*{Step 3: the curve $C$ and its embedding in $\Lambda$}
Base change of $\tgen\otimes_KF'\cong\tau_0\otimes_EF'$ along $F'\subseteq\tilde F$ gives
\begin{equation}\label{eq:tauLambda}
\tau_0\otimes_E\tilde F\cong\tgen\otimes_K\tilde F\cong\Lambda .
\end{equation}
So $\tau_0$ is non-trivial: as $|G|$ is prime, it is $\Spec M_0$ for a cyclic field extension $M_0/E$ of degree $p$ with $G=\Gal(M_0/E)$.
\begin{itemize}
\item $\trdeg_kE=1$. If $\trdeg_kE=0$ then $E=k_0$, $\tau_0\otimes_{k_0}\kb$ is split, and so is $\tau_0\otimes_E\tilde F$, contradicting \eqref{eq:tauLambda}.
\item $k_0$ is algebraically closed in $M_0$. Suppose not, and let $k_0'\supsetneq k_0$ be the algebraic closure of $k_0$ in $M_0$. It is $G$-stable and $k_0'\cap E=k_0$, so $G$ acts non-trivially, hence faithfully, on $k_0'$; thus $[k_0':k_0]=p$ and $(k_0')^G=k_0$. As $E$ is regular over $k_0$, $Ek_0'\cong E\otimes_{k_0}k_0'$ has degree $p$ over $E$, so $M_0=Ek_0'\cong E\otimes_{k_0}k_0'$ as $G$-algebras. Then $\tau_0\otimes_E\tilde F\cong\tilde F\otimes_{k_0}k_0'$ is split, as $k_0'\otimes_{k_0}\kb\cong\kb^p$; this contradicts \eqref{eq:tauLambda}.
\end{itemize}
So $M_0$ is the function field of a smooth projective geometrically integral curve $C$ over $k_0$, on which $G$ acts faithfully by $k_0$-automorphisms. Let $D=C/G$, with function field $E$; it is smooth and geometrically integral as $E/k_0$ is regular. The genera $g(C)$, $g(D)$ are unchanged by extension of the base field to $\kb$, since smoothness and $h^1(\mathcal O)$ are preserved by flat base change \cite[III.9.3]{Hart77}. The field $\tilde M_0:=M_0\otimes_{k_0}\kb$ is the function field of $C_{\kb}$, with $G$ acting through $M_0$, and $\tilde M_0^G=E\otimes_{k_0}\kb=\tilde E$.

The embedding. Since $M_0\otimes_E\tilde E=M_0\otimes_E(E\otimes_{k_0}\kb)=\tilde M_0$, the left side of \eqref{eq:tauLambda} is $\tilde M_0\otimes_{\tilde E}\tilde F$, and \eqref{eq:tauLambda} is a $G$-equivariant $\tilde F$-algebra isomorphism $\iota\colon\tilde M_0\otimes_{\tilde E}\tilde F\to\Lambda$. Its restriction to $\tilde M_0\otimes1$ is a $G$-equivariant embedding of fields over $\tilde E$. We identify $\tilde M_0$ with its image: a $G$-stable subfield of $\Lambda\subseteq\Om$ containing $\tilde E$, with $\tilde M_0/\tilde E$ Galois with group $G$. (Only $\tilde L\otimes_{\tilde K}\tilde F$ is known to be a field, by Step~2; the argument does not use $L\otimes_KF'$.)

\subsubsection*{Step 4: $C_{\kb}$ has a $G$-fixed point}
Apply Proposition~\ref{prop:A4} over the algebraically closed field $\kb$ to the generic torsor $\tgen\otimes_K\tilde K$ of $V_{\kb}$ (Lemma~\ref{lem:regular}(c)), with total space $\tilde L\subseteq\Om$, to $A=G$, to the valuation $\nu$ of Lemma~\ref{lem:flag} for $V_{\kb}$ ($\Om$ is an algebraic closure of $\tilde L$), and to $\tilde F\supseteq\tilde K$ and $\tilde E\subseteq\tilde F$: by $\iota$, $\tgen\otimes_K\tilde F\cong(\tilde M_0/\tilde E)\otimes_{\tilde E}\tilde F$, and the torsor $\tilde M_0/\tilde E$ is classified by $\psi(\gamma)=\gamma|_{\tilde M_0}$ for $\gamma\in\Gam_{\tilde E}$. Then $A\cap\varphi(\Dec_{\tilde F})$ has index $\le d''<p$ in $A=\Z/p$, so it equals $G$, and Proposition~\ref{prop:A4}(b) gives $\psi(\Dec_{\tilde E})=G$. Thus every element of $G$ is the restriction of some $\gamma\in \Dec_{\tilde E}$, and so preserves the valuation $\nu|_{\tilde M_0}$.

The valuation $\nu|_{\tilde M_0}$ is trivial on $\kb$ and non-trivial on $\tilde E$: otherwise the residue map would embed $\tilde E$, of transcendence degree $1$ over $\kb$, into $\kappa(\tilde E)=\kb$ (Lemma~\ref{lem:flag}(i)). Let $w=\nu|_{\tilde M_0}$. It is non-trivial on $\kb(C)$ (it is non-trivial on $\tilde E\subseteq\tilde M_0$) and trivial on $\kb$. Abhyankar's inequality (Lemma~\ref{lem:abhyankar}) gives $\dim_\Q w(\kb(C)^\times)\otimes\Q+\trdeg_{\kb}\kappa(w)\le1$, so the rational rank is $1$ and $\kappa(w)$ is algebraic over $\kb$, hence equal to $\kb$; a valuation of a function field of one variable over $\kb$ that is trivial on $\kb$ is discrete \cite[Ch.~1]{St}, and its valuation ring is the local ring of a unique closed point $P\in C(\kb)$ \cite[I, \S6]{Hart77}.
Since every element of $G$ maps $\mathcal O_P$ to itself, $G$ fixes $P$.

\subsubsection*{Step 5: at least $[k_0(\zeta):k_0]$ fixed points}
Let $\Fix\subseteq C(\kb)$ be the set of $G$-fixed points; it is finite, and non-empty by Step~4.
\begin{enumerate}[label=(\arabic*)]
\item The rotation character is faithful. For $P\in\Fix$, $G$ acts on $\mathfrak m_P/\mathfrak m_P^2$ by a character $\theta_P\colon G\to\kb^\times$. In characteristic $0$, averaging
\[
t'=\frac1p\sum_{h\in G}\theta_P(h)^{-1}h(t)
\]
for a uniformiser $t$ gives a uniformiser $t'\equiv t\bmod\mathfrak m_P^2$ with $h(t')=\theta_P(h)t'$, so $G$ acts on $\widehat{\mathcal O}_P=\kb[[t']]$ through $\theta_P$. If $\theta_P(h)=1$ then $h$ fixes $\mathcal O_P\subseteq\widehat{\mathcal O}_P$, hence $\kb(C)$, so $h=1$. Thus $\theta_P(\sigma)$ is a primitive $p$-th root of unity.
\item Galois compatibility. Let $\Gam_{k_0}$ act (through its quotient $\Gal(\kb/k_0)$) on $\kb(C)=M_0\otimes_{k_0}\kb$ by $1\otimes s$ (semilinear over $\kb$, trivial on $M_0$), and on $C(\kb)$ by letting $sP$ be the point with $\mathcal O_{sP}=s(\mathcal O_P)$, so $\mathfrak m_{sP}=s(\mathfrak m_P)$. This action commutes with $G$, which acts through $M_0$; hence $\Fix$ is $\Gam_{k_0}$-stable. Applying $s$ to $h(t)\equiv\theta_P(h)t\pmod{\mathfrak m_P^2}$ and using $s\circ h=h\circ s$ gives $h(st)\equiv s(\theta_P(h))\,st\pmod{\mathfrak m_{sP}^2}$ with $st$ a uniformiser at $sP$, i.e.\ $\theta_{sP}(h)=s(\theta_P(h))$.
\item Counting. So $P\mapsto\theta_P(\sigma)$ is a $\Gam_{k_0}$-equivariant map from $\Fix$ to the set of primitive $p$-th roots of unity, with non-empty, $\Gam_{k_0}$-stable image. $\Gam_{k_0}$ acts on that set through $\Gal(k_0(\zeta)/k_0)\le(\Z/p)^\times$, freely, so every orbit has exactly $[k_0(\zeta):k_0]$ elements. The map need not be injective, but its image is a non-empty union of orbits, so $|\Fix|\ge|\text{image}|\ge[k_0(\zeta):k_0]$. Hence
\[
r:=|\Fix|\ge[k_0(\zeta):k_0].
\]
\end{enumerate}

\subsubsection*{Step 6: Riemann--Hurwitz}
$C_{\kb}\to D_{\kb}$ is Galois of prime degree $p$. Every stabiliser is $1$ or $G$, the ramification points are exactly the points of $\Fix$, each with ramification index $p$, and all ramification is tame (characteristic $0$). So \cite[IV.2.4]{Hart77}
\begin{equation}\label{eq:RH}
2g(C)-2=p\,(2g(D)-2)+r(p-1),\qquad\text{i.e.}\qquad g(C)=p\,g(D)-p+1+\tfrac{r(p-1)}2.
\end{equation}

\subsubsection*{Step 7: a generic line}
Write $\tilde L=\kb(x_1,\dots,x_n)$ with $x_1,\dots,x_n$ linear coordinates on $V$, and pick $u\in\tilde M_0\smallsetminus\kb$; it is transcendental over $\kb$. Here the embedding $\tilde M_0\subseteq\Lambda$ of Step~3 is used: $u$ is an element of $\Lambda$, and $[\Lambda:\tilde L]=d''$ by Step~2.
\begin{enumerate}[label=(\arabic*)]
\item Differentials. $\Lambda/\kb(x)$ is finite separable, so $dx_1,\dots,dx_n$ is a basis of the module of K\"ahler differentials $\Omega^1_{\Lambda/\kb}$. In characteristic $0$, elements $z_1,\dots,z_s\in\Lambda$ are algebraically independent over $\kb$ iff $dz_1,\dots,dz_s$ are $\Lambda$-linearly independent \cite[\S26]{Ma}, \cite[\S16.5--16.6]{Ei}. So $du=\sum_ia_i\,dx_i$ with $(a_1,\dots,a_n)\ne0$ in $\Lambda^n$.
\item Choosing the line direction. For linearly independent linear forms $y_2,\dots,y_n$ with coefficient rows $\ell_2,\dots,\ell_n\in\kb^n$,
\[
dy_2\wedge\dots\wedge dy_n\wedge du=\Big(\sum_i\lambda_ia_i\Big)\,dx_1\wedge\dots\wedge dx_n,
\]
where $\lambda\in\kb^n\smallsetminus0$ is the vector of signed maximal minors of the matrix with rows $\ell_2,\dots,\ell_n$; it spans the orthogonal of $\operatorname{span}(\ell_j)$, and every $\lambda\ne0$ arises in this way up to a scalar. The set $\{\lambda\in\kb^n:\sum\lambda_ia_i=0\}$ is a proper $\kb$-subspace, as some $a_i\ne0$. So we can choose $y_2,\dots,y_n$ with $dy_2,\dots,dy_n,du$ linearly independent, i.e.\ $y_2,\dots,y_n,u$ algebraically independent over $\kb$, and complete them to linear coordinates $y_1,\dots,y_n$ on $V$. (If $n=1$ this step is empty.)
\item Base change to $k_1$. Let $k_1$ be the algebraic closure of $\kb(y_2,\dots,y_n)$ in $\Om$; it is algebraically closed of characteristic $0$.
\begin{enumerate}[label=(\alph*)]
\item $y_1$ is transcendental over $k_1$. The field $\Lambda_1:=\Lambda k_1$ satisfies $[\Lambda_1:k_1(y_1)]\le[\Lambda:\tilde L]=d''$ by Lemma~\ref{lem:compositum}, since $k_1(y_1)=\tilde Lk_1$. So $\Lambda_1$ is a function field of one variable over $k_1$, with constant field $k_1$, and $y_1\in\Lambda_1$ has degree $\delta:=[\Lambda_1:k_1(y_1)]\le d''$.
\item $y_2,\dots,y_n,u$ are algebraically independent over $\kb$ (Step~7(2)), so $u$ is transcendental over $k_1$, the algebraic closure of $\kb(y_2,\dots,y_n)$. As $\tilde M_0$ is algebraic over $\kb(u)$, $\{u\}$ is a transcendence basis of $\tilde M_0/\kb$ that stays algebraically independent over $k_1$; thus $\tilde M_0$ and $k_1$ are free over $\kb$. They are linearly disjoint over $\kb$: $\kb(u)\otimes_{\kb}k_1$ embeds in $k_1(u)$, and $\tilde M_0\otimes_{\kb(u)}k_1(u)$ is a domain, being a localisation of the coordinate ring of the integral curve $C_{k_1}$ (a variety over $\kb=\bar\kb$ stays integral after any field extension), hence a field since it is finite over $k_1(u)$, so it maps injectively to $\Om$. Thus $M_1:=\tilde M_0k_1\cong\Frac(\tilde M_0\otimes_{\kb}k_1)$ is the function field of $C_{k_1}$, of genus $g(C)$; $G$ acts $k_1$-linearly on it through $\tilde M_0$, and $M_1^G=\tilde Ek_1$ is the function field of $D_{k_1}$, of genus $g(D)$. Moreover $M_1\subseteq\Lambda_1$ and $[\Lambda_1:M_1]<\infty$, both being function fields of one variable over $k_1$.
\end{enumerate}
\end{enumerate}

\subsubsection*{Step 8: the gonality descends}
(Compare \cite[Prop.~1.1(vii)]{Po} and \cite[Lemma~2.4]{FW25}; we argue directly.) Let $\chi(T)\in M_1[T]$ be the characteristic polynomial of multiplication by $y_1$ on the $M_1$-vector space $\Lambda_1$, of dimension $e_1:=[\Lambda_1:M_1]$. Then $\chi\notin k_1[T]$: otherwise $\chi(y_1)=0$ (Cayley--Hamilton) would make $y_1$ algebraic over $k_1$. If $\chi(b)\in k_1$ for $e_1+1$ distinct $b\in k_1$, Vandermonde inversion would put all coefficients of $\chi$ in $k_1$. So there is $b\in k_1$ with
\[
\eta:=(-1)^{e_1}\chi(b)=N_{\Lambda_1/M_1}(y_1-b)\notin k_1 .
\]
For a place $P$ of $M_1$, $v_P(N_{\Lambda_1/M_1}(z))=\sum_{Q\mid P}f(Q|P)\,v_Q(z)$ \cite[I~\S4]{Se2}, and $f(Q|P)=1$ since $k_1$ is algebraically closed. Hence the pole divisor of $\eta$ satisfies
\[
(\eta)_\infty\le\sum_P\sum_{Q\mid P}\max\bigl(0,-v_Q(y_1-b)\bigr)\,P,
\]
whose degree is $\deg(y_1-b)_\infty=\delta$. So $[M_1:k_1(\eta)]=\deg(\eta)_\infty\le\delta\le d''$ \cite[Thm.~1.4.11]{St} (here $k_1$ is algebraically closed, hence the full constant field of $M_1$).

\subsubsection*{Step 9: Castelnuovo's inequality}
Let $\mathcal F$ be a function field of one variable over an algebraically closed field, with subfields $\mathcal F_1,\mathcal F_2$ containing the constants such that $\mathcal F=\mathcal F_1\mathcal F_2$ and $[\mathcal F:\mathcal F_i]=n_i<\infty$. Then Castelnuovo's inequality \cite[Thm.~3.11.3]{St} (the constant field is algebraically closed, hence the full constant field of every field involved)
\[
g(\mathcal F)\le n_1g(\mathcal F_1)+n_2g(\mathcal F_2)+(n_1-1)(n_2-1).
\]
Take $\mathcal F=M_1$, $\mathcal F_1=M_1^G$ ($n_1=p$, genus $g(D)$) and $\mathcal F_2=k_1(\eta)$ (genus $0$, $n_2=[M_1:k_1(\eta)]\le d''$).

Non-factoring. $\mathcal F_1\mathcal F_2$ lies between $\mathcal F_1$ and $M_1$, and $[M_1:\mathcal F_1]=p$ is prime. If $\mathcal F_1\mathcal F_2=\mathcal F_1$, then $k_1(\eta)\subseteq\mathcal F_1$ and $n_2=p\,[\mathcal F_1:k_1(\eta)]\ge p>d''$, a contradiction. So $\mathcal F_1\mathcal F_2=M_1$, and
\[
g(C)\le p\,g(D)+(p-1)(n_2-1).
\]

\subsubsection*{Step 10: arithmetic}
Substituting \eqref{eq:RH},
\[
p\,g(D)-p+1+\tfrac{r(p-1)}2\le p\,g(D)+(p-1)(n_2-1)\iff\tfrac{r(p-1)}2\le(p-1)n_2\iff r\le2n_2.
\]
Combining with Step~5 and $n_2\le d''\le d/m_0$,
\[
[k_0(\zeta):k_0]\le r\le\frac{2d}{m_0}\quad\Longrightarrow\quad d\ge\tfrac12\,m_0\,[k_0(\zeta):k_0]=\tfrac12[k_0(\zeta):k]\ge\tfrac12[k(\zeta):k]=\tfrac m2 .
\]
So every $d$ with $\edd d_k(\tgen)\le1$ satisfies $d\ge\lceil\frac m2\rceil$, which is the lower bound of Theorem~\ref{thm:Zp}.

\begin{proof}[Proof of Theorem~\ref{thm:Zp}]
The upper bound is Corollary~\ref{cor:pertorsor} and the lower bound is Steps~1--10. If $m$ is even, Lemma~\ref{lem:cyclo} gives $[k(c):k]=m/2=\lceil m/2\rceil$. For $k=\Q$, $m=p-1$ is even.
\end{proof}

\subsection{Sharpness of the roots-of-unity bound}\label{ss:sharpB}
For $G=\Z/p$ the roots-of-unity bound of Corollary~\ref{cor:rou} is $d_{1,k}(\tgen)\le[k(\mu_p):k]=m$, since $\ed_{k(\zeta)}(\Z/p)=1$ (Example~\ref{ex:cyc}). Theorem~\ref{thm:Zp} decides when it is sharp, except in one family of cases.

\begin{corollary}[sharpness of Corollary~\ref{cor:rou} for $\Z/p$]\label{cor:sharp}
Let $\operatorname{char}k=0$.
\begin{enumerate}[label=\textup{(\roman*)}]
\item If $m=1$, then $d_{1,k}(\tgen)=1=m$: Corollary~\ref{cor:rou} is sharp.
\item If $m\ge2$ and $-1\in \Theta$ (equivalently $m$ is even, equivalently $\zeta\notin k(c)$), then $d_{1,k}(\tgen)=m/2$: Corollary~\ref{cor:rou} overshoots by exactly a factor $2$, and the optimum is attained by the constant-field extension $Kk(c)$.
\item If $m\ge2$ and $-1\notin \Theta$ (equivalently $m$ is odd and $m\ge3$; then $k(c)=k(\zeta)$), then $\frac{m+1}2\le d_{1,k}(\tgen)\le m$, and it is not known whether Corollary~\ref{cor:rou} is sharp (Question~\ref{q:sharpodd}).
\end{enumerate}
\end{corollary}
\begin{proof}
(i) $1\le d_{1,k}(\tgen)\le m$. (ii) Lemma~\ref{lem:cyclo} and Theorem~\ref{thm:Zp}, the witness being that of Corollary~\ref{cor:pertorsor}. (iii) By Lemma~\ref{lem:cyclo}, $k(c)=k(\zeta)$, so Theorem~\ref{thm:Zp} reads $\lceil m/2\rceil\le d_{1,k}(\tgen)\le m$, and $\lceil m/2\rceil=(m+1)/2$ for odd $m$.
\end{proof}

In case (iii) no constant-field extension does better than $Kk(\zeta)$: if $[k':k]<m$ then $c\notin k'$, as $k(c)=k(\zeta)$ has degree $m$, so $\ed_{k'}(\Z/p)\ge2$ by Remark~\ref{rem:checksB}(i), and $\ed_k(\tgen\otimes Kk')=\ed_{k'}(\Z/p)$ by Lemma~\ref{lem:regular}(c) and Proposition~\ref{prop:bfe}. The question is therefore whether an extension of $K$ that is not a constant-field extension can do better.

\begin{question}\label{q:sharpodd}
Let $\operatorname{char}k=0$ and let $m=[k(\zeta_p):k]$ be odd, $m\ge3$. Is $d_{1,k}(\tgen(\Z/p))=m$? The smallest case is $p=7$, $k=\Q(\sqrt{-7})$: this is the quadratic subfield of $\Q(\zeta_7)$ (as $7\equiv3\bmod4$), so $m=3$, and $d_{1,k}(\tgen)\in\{2,3\}$.
\end{question}

\begin{remark}[where the method stops for $m$ odd]\label{rem:stalls}
For an extension $F'/K$ with $[F':K]=d$ that is not a constant-field extension, the argument gives a $G$-curve with at least $m$ fixed points (one Galois orbit, Step~5) and Riemann--Hurwitz and Castelnuovo give $r\le2d$ (Step~10), hence $d\ge m/2$, i.e.\ $d\ge(m+1)/2$ for $m$ odd. To exclude $m/2<d<m$ one would need either more fixed points than one Galois orbit, or a sharper genus bound than Castelnuovo's for curves whose function field is generated by functions of degree $\le d$ over a base of genus $g(D)$; we know of no way to obtain either from non-constant quadratic extensions. For instance, at $p=7$, $k=\Q(\sqrt{-7})$ ($m=3$) the method excludes $d=1$ but not $d=2$.
\end{remark}

\begin{remark}[checks]\label{rem:checksB}
(i) For $d=1$ (so $m_0=1$, $d''=1$) Steps~1--10 give $r\le2$, hence $m\le2$ and $c\in k$ by Lemma~\ref{lem:cyclo}: if $\ed_k(\Z/p)=1$ then $\zeta+\zeta^{-1}\in k$. With Lemma~\ref{lem:B1} this is the criterion $\ed_k(\Z/p)=1\iff\zeta+\zeta^{-1}\in k$ in characteristic $0$.
(ii) Over $\Q$: for $p=3$, $c=-1\in\Q$ and $d_{1,\Q}(\tgen)=1$, i.e.\ $\ed_\Q(\Z/3)=1$. For $p=5$, $d_{1,\Q}(\tgen)=2$: thus $\ed_\Q(\Z/5)\ge2$, and the profile drops to $1$ at degree $2$. For $p=7$, $d_{1,\Q}(\tgen)=3$. The role of the constant field $k_0$ is visible at $p=7$: for $F'=K(\sqrt{-7})=Kk_0$ with $k_0=\Q(\sqrt{-7})$ one has $d=2$, $m_0=2$, and Step~10 would require $[k_0(\zeta_7):k_0]=3\le2d/m_0=2$; so $\ed_\Q(\tgen\otimes K(\sqrt{-7}))\ge2$.
\end{remark}

% ---- end thmB.tex ----

% ---- begin S5.tex ----
\section{The symmetric group \texorpdfstring{$S_5$}{S5}: the first jump degree is \texorpdfstring{$4$}{4}}\label{sec:S5}

Throughout, $G=S_5$, $V$ is a faithful $k$-representation, $L=k(V)$, $K=L^G$, $\tau=\tgen=\Spec L\to\Spec K$, and $F_1:=L^{A_5}=K(\delta)$ with $\delta^2=\Delta\in K$, so that $g(\delta)=\operatorname{sgn}(g)\delta$ for $g\in G$. Lemmas~\ref{lem:S5bc}--\ref{lem:S5free} are stated for an arbitrary finite group $G$ (and $k$ of characteristic $0$ in Lemma~\ref{lem:S5bc}).

\begin{theorem}\label{thm:S5}
Let $k$ be a field of characteristic $0$. Then $\edd d_k(\tgen)\ge2$ for $d\le3$ and $\edd4_k(\tgen)\le1$. Hence $d_1(\tgen(S_5))=4$. More precisely, $\edd d_k(\tgen)=1$ for $4\le d\le119$ and $\edd d_k(\tgen)=0$ for $d\ge120$; and $\edd d_k(\tgen)=2$ for $d\le3$, the lower bound being Proposition~\ref{prop:S5lower} and the upper bound $\ed_k(S_5)\le2$ \cite[Thm.~6.5(d)]{BR97}, valid over every $k$ of characteristic $0$.
\end{theorem}

\subsection{Preliminaries}

\begin{remark}[the values of $\ed_k$ over a non-closed field]\label{rem:BRk}
We use $\ed_k(S_4)=\ed_k(S_5)=2$, $\ed_k(S_6)=3$ and $\ed_k(A_5)=2$ for every field $k$ of characteristic $0$ (for $A_5$ with $\sqrt5\in k$). These are \cite[Thm.~6.5 and Thm.~6.7(b)]{BR97}, which hold over every field of characteristic $0$ (the bound $\ed_k(S_n)\le n-3$ for $n\ge5$ is proved there by cross-ratios). For $A_5$ with $\sqrt5\in k$ the upper bound can also be seen directly: the icosahedral representation of $A_5$ over $k$ is faithful of dimension $3$, so $A_5$ acts generically freely on $\PP^2$, whose twist by any torsor is again $\PP^2$ and has points in the free locus over every infinite field, and Theorem~\ref{thm:geom} gives $\ed_k(A_5)\le2$. For $A_5$ and $S_4$ the lower bound $2$ also follows from Theorem~\ref{thm:rankindex} with $A=V_4$.
\end{remark}

\begin{lemma}[base change to $\bar k$]\label{lem:S5bc}
Let $k$ have characteristic $0$. Let $\bar L=\bar k(V)$, $\bar K=\bar L^G$, $\bar\tau=\Spec\bar L$. If $\edd d_k(\tau)\le j$ then $\edd d_{\bar k}(\bar\tau)\le j$. The same holds with ``$\le j$'' replaced by ``$=0$''.
\end{lemma}
\begin{proof}
$L\otimes_K\bar K\to\bar L$ is an equivariant map of $G$-torsors over $\bar K$, hence an isomorphism by (T2); so $\bar\tau=\tau\otimes_K\bar K$. Let $F'/K$ have degree $d'\le d$, with $\tau\otimes F'\cong P_0\otimes_EF'$ and $\operatorname{trdeg}_kE\le j$. As $F'/K$ is separable, $F'\otimes_K\bar K$ is a product of fields; let $F''$ be one factor, so $[F'':\bar K]\le d$. Let $E'=E\bar k\subseteq F''$. Then $\operatorname{trdeg}_{\bar k}E'\le\operatorname{trdeg}_kE\le j$ and $\bar\tau\otimes_{\bar K}F''=\tau\otimes_KF''\cong(P_0\otimes_EE')\otimes_{E'}F''$.
\end{proof}

\begin{lemma}[the last jump]\label{lem:S5last}
Let $k$ have characteristic $0$ and $G\ne1$ be finite. Then $\edd d_k(\tgen)\ge1$ for $d<|G|$, and $\edd{|G|}_k(\tgen)=0$.
\end{lemma}
\begin{proof}
This is Proposition~\ref{prop:d0}(c) ($k$ is perfect, being of characteristic $0$): $d_0(\tgen)=|G|$.
\end{proof}

\begin{lemma}[free points]\label{lem:S5free}
Let $X$ be a projective $k$-scheme of pure dimension $1$ with an action of $G$ by $k$-automorphisms, $X_{\free}\subseteq X$ the open subscheme on which $G$ acts freely on geometric points, and $\pi:X_{\free}\to Y=X_{\free}/G$ the quotient torsor. If $P$ is a $G$-torsor over a field $F\supseteq k$ and there is a $G$-equivariant $k$-morphism $f:\Spec P\to X_{\free}$, then $\ed_k(P)\le1$.
\end{lemma}
\begin{proof}
$\pi f$ is $G$-invariant, so it factors through $\Spec F\xrightarrow{y}Y$, and $f$ is a morphism of $G$-torsors $P\to y^*X_{\free}$, an isomorphism by (T2). With $E=\kappa(y_0)\subseteq F$, $y_0$ the image of $y$, we get $P\cong\pi^{-1}(y_0)\otimes_EF$ and $\operatorname{trdeg}_kE\le\dim Y=1$.
\end{proof}

For a $G$-torsor $\Spec N$ over a field $F_0$ the twist ${}^NX=(X\times_kN)/G$ exists since $X$ is projective, and for $F\supseteq F_0$ the set ${}^NX(F)$ is the set of $G$-equivariant $k$-morphisms $\Spec(N\otimes_{F_0}F)\to X$. The open subscheme ${}^N(X_{\free})$ corresponds to maps into $X_{\free}$; it becomes $X_{\free}\times_kN$ over $N$.

\subsection{The lower bound}

In this subsection $k=\bar k$ has characteristic $0$.

\begin{lemma}[faithful $S_5$-curves]\label{lem:S5curves}
Let $k=\bar k$ have characteristic $0$.
\begin{enumerate}[label=\textup{(\roman*)}]
\item If $S_5$ acts faithfully on a smooth projective curve, every $A_5$-orbit has at least $12$ points.
\item Every homomorphism $S_5\to\PGL_2(k)$ has kernel containing $A_5$.
\item $S_5$ has no faithful action on a curve of genus $0$ or $1$.
\end{enumerate}
\end{lemma}
\begin{proof}
(i) A point stabiliser acts faithfully on the cotangent line: in characteristic $0$, if a finite-order automorphism fixes $x$ and is trivial on $\mathfrak m_x/\mathfrak m_x^2$, averaging a uniformiser gives a fixed uniformiser, so it is the identity. Stabilisers are therefore cyclic; cyclic subgroups of $A_5$ have order $\le5$, so orbits have size $\ge60/5=12$.
(ii) The kernel is normal, so it is $1$, $A_5$ or $S_5$. If it were $1$, $S_5$ would be a subgroup of order $120$ of $\PGL_2(k)$; the finite subgroups of $\PGL_2(k)$ are cyclic, dihedral, $A_4$, $S_4$, $A_5$, so it would be dihedral of order $120$ and contain an element of order $60$, whereas elements of $S_5$ have order $\le6$.
(iii) Genus $0$ is (ii). For genus $1$, $\Aut(C)=C\rtimes\Aut(C,O)$ with $\Aut(C,O)$ cyclic, so a finite subgroup is an extension of a cyclic group by an abelian group, hence solvable.
\end{proof}

\begin{proposition}\label{prop:S5lower}
Let $k=\bar k$ have characteristic $0$ and $[F':K]\le3$. Then $\ed_k(\tau\otimes F')\ge2$.
\end{proposition}
\begin{proof}
Suppose $\ed_k(\tau\otimes F')\le1$ and take $N_0,H$ as in Section~\ref{ss:tors}, so $[G:H]$ divides $[F':K]\le3$. $S_5$ has no subgroup of index $3$ (the action on three cosets would have kernel in $\{1,A_5,S_5\}$ and image of order $\le2$, which is not transitive); so $H=S_5$ or $H=A_5$; in particular $H\ne1$.

Case $H=A_5$. Then $[F':K]=2$ and $N_0=L$, so $F'\to L$ identifies $F'$ with the unique quadratic subfield $F_1=L^{A_5}$, and $\tau\otimes_KF_1=\Ind_{A_5}^{S_5}(L/F_1)$. Here $L/F_1$ is the generic $A_5$-torsor of $V|_{A_5}$, so Lemma~\ref{lem:ind} gives $\ed_k(\tau\otimes F_1)=\ed_k(L/F_1)=\ed_k(A_5)\ge\rk V_4=2$ by Theorem~\ref{thm:rankindex} with $d=1$ and $A=V_4\le A_5$ (this is also the lower bound of Reichstein--Youssin \cite{RY}, so this case does not depend on the new theorem). Contradiction.

Case $H=G$. Then $N_0=L\otimes_KF'$ is a field of degree $e=[F':K]\in\{1,2,3\}$ over $L$. Proposition~\ref{prop:witness} gives a faithful $S_5$-curve $C$ and a $G_C$-invariant class $D$ of degree $e$ with $r=\dim|D|\ge1$.
If $r\ge2$: for special $D$, Clifford gives $2r\le e\le3$, impossible; so $D$ is non-special, $r=e-g(C)$ and $g(C)\le1$, contradicting Lemma~\ref{lem:S5curves}(iii).
If $r=1$ (here only (i) and (ii) of Lemma~\ref{lem:S5curves} are used, not the genus): $G_C$ acts on $|D|\cong\PP^1$ by $E\mapsto h^*E$; this action is projective-linear because $h^*\mathcal O(D)\cong\mathcal O(D)$. By Lemma~\ref{lem:S5curves}(ii) $A_5$ acts trivially, so every effective $E\in|D|$ is $A_5$-invariant; its support is a union of $A_5$-orbits, each of size $\ge12$, while $\deg E=e\le3$. Contradiction.
\end{proof}

\paragraph{Larger symmetric groups.}
For $n\ge6$ the lower bound $d_1(\tgen(S_n))\ge4$ already follows from Sylow divisibility.

\begin{lemma}\label{lem:edqSn}
Let $\operatorname{char}k=0$, $q$ a prime and $n\ge q$. Then $\ed_{k,q}(S_n)\ge\lfloor n/q\rfloor$.
\end{lemma}
\begin{proof}
Let $m=\lfloor n/q\rfloor$; $S_n$ contains $(\Z/q)^m$, generated by $m$ disjoint $q$-cycles. (The result is also in \cite[Cor.~4.2]{MR09}.) For $H\le G$ and an $H$-torsor $\tau_H$ over $F$, $\Ind_H^G(\tau_H\otimes F')=(\Ind_H^G\tau_H)\otimes F'$, so Lemma~\ref{lem:ind} gives $\ed_k((\Ind_H^G\tau_H)\otimes F')=\ed_k(\tau_H\otimes F')$ for all finite separable $F'/F$ of degree prime to $q$; hence $\ed_{k,q}(G)\ge\ed_{k,q}(H)$. By Lemma~\ref{lem:qinv} with $k'=k(\mu_q)$ (separable of degree dividing $q-1$) and Theorem~\ref{thm:KM} over $k'$, $\ed_{k,q}((\Z/q)^m)=\ed_{k',q}((\Z/q)^m)=\rdim_{k'}((\Z/q)^m)=m$.
\end{proof}

\begin{corollary}\label{cor:SnSylow}
Let $\operatorname{char}k=0$ and $n\ge4$. Every prime $q\le n/2$ divides $d_1(\tgen(S_n))$. In particular $2\mid d_1(\tgen(S_n))$ for $n\ge4$, and $6\mid d_1(\tgen(S_n))$, so $d_1(\tgen(S_n))\ge6$, for $n\ge6$.
\end{corollary}
\begin{proof}
Lemma~\ref{lem:edqSn} gives $\ed_{k,q}(S_n)\ge2>1$ for $q\le n/2$, and Theorem~\ref{thm:div}(b) with $j=1$ gives $q\mid d_1$.
\end{proof}
For $n=5$ only $q=2$ qualifies, so the geometric argument of Proposition~\ref{prop:S5lower} is needed (see Remark~\ref{rem:Sngeneral}).

\begin{remark}\label{rem:Sngeneral}
The proof of Proposition~\ref{prop:S5lower} uses only that $A_n$ is simple and non-abelian (so a non-trivial action on an orbit of at most $3$ points forces a fixed point, and point stabilisers on faithful curves are cyclic) and that $S_n$ is neither a subgroup of $\PGL_2$ nor of the automorphism group of an elliptic curve (so a faithful $S_n$-curve has genus $\ge2$). It therefore gives $d_1(\tgen(S_n))\ge4$ for every $n\ge5$, independently of Karpenko--Merkurjev; for $n\ge6$ Corollary~\ref{cor:SnSylow} is stronger.
\end{remark}

\subsection{The upper bound}

\paragraph{The icosahedral conic.}
Let $\phi=(1+\sqrt5)/2$. The rotation group $A_5$ of the icosahedron with vertices the cyclic permutations of $(0,\pm1,\pm\phi)$ consists of orthogonal matrices permuting the vertices. If $v_1,v_2,v_3\in\Q(\sqrt5)^3$ are linearly independent vertices and $R$ is a rotation, then $R=[Rv_1\,Rv_2\,Rv_3][v_1\,v_2\,v_3]^{-1}$ has entries in $\Q(\sqrt5)$. This gives an absolutely irreducible $\rho:A_5\hookrightarrow\operatorname{SO}_3(\Q(\sqrt5))$ preserving $q=x_1^2+x_2^2+x_3^2$; since $\Hom_{A_5}(\rho,\rho^\vee)$ is one-dimensional, the invariant quadratic forms are $k\cdot q$. Let $Q=\{q=0\}\subset\PP^2$. Over $\bar k$, $Q\cong\PP^1$ and $A_5$ acts faithfully through $\operatorname{SO}_3\cong\PGL_2$ (an element of $\operatorname{SO}_3$ fixing the conic pointwise is a scalar, hence $1$), hence freely off a finite set of points.

\begin{proposition}\label{prop:S5upper}
For every field $k$ of characteristic $0$ there is a field extension $F_2/K$ of degree $4$ containing $K(\sqrt{5\Delta})$ with $\ed_k(\tau\otimes F_2)\le1$.
\end{proposition}
\begin{proof}
Case $\sqrt5\in k$ (Klein). Then $A_5$ acts on the $k$-conic $Q$. Let $T_1=L/F_1$ be the $A_5$-torsor. The twist ${}^{T_1}Q$ is a conic over $F_1$; a rational line meets it in a subscheme of degree $2$ (it is a smooth conic in $\PP^2$), so it has a point over a quadratic extension $F_2/F_1$ and ${}^{T_1}Q\cong\PP^1_{F_2}$. The free locus is non-empty and open and $F_2$ is infinite, so it has an $F_2$-point. By Lemma~\ref{lem:S5free}, $\ed_k(T_1\otimes F_2)\le1$. As $\tau\otimes_KF_1=\prod_{S_5/A_5}L=\operatorname{Ind}_{A_5}^{S_5}T_1$, (T3) gives $\ed_k(\tau\otimes F_2)\le1$, and $[F_2:K]=4$. Here $K(\sqrt{5\Delta})=F_1$.

Case $\sqrt5\notin k$. Let $k'=k(\sqrt5)$ with generator $\sigma$ of $\Gal(k'/k)$.

Step 1: an $S_5$-action on a $k$-curve. Let $X=Q_{k'}$, viewed as a $k$-scheme; $X_{\bar k}=Q\sqcup Q^\sigma$. The two $3$-dimensional characters of $A_5$ are swapped both by $\sigma$ and by conjugation by a transposition $t$, so $\rho^\sigma\cong\rho\circ c_t$ over $k'$: there is $U\in\operatorname{GL}_3(k')$ with $\rho^\sigma(a)=U\rho(tat)U^{-1}$. Let $\mathsf S(x)=U^{-1}\sigma(x)$, a $\sigma$-semilinear automorphism of $\PP^2_{k'}$ with $\mathsf S^{-1}(y)=\sigma(U)\sigma(y)$. Then
\[\mathsf S\rho(a)\mathsf S^{-1}=U^{-1}\rho^\sigma(a)U=\rho(tat).\]
Also $\mathsf S^2(x)=(\sigma(U)U)^{-1}x$, and applying $\sigma$ to the defining relation gives
\[\rho(a)=(\sigma(U)U)\rho(a)(\sigma(U)U)^{-1};\]
by absolute irreducibility $\sigma(U)U$ is a scalar, so $\mathsf S^2=\id$ on $\PP^2$. The conic $\mathsf S(Q)$ is stable under $\rho(A_5)$ (use $\rho(a)\mathsf S=\mathsf S\rho(tat)$), so its equation is a semi-invariant quadratic form, hence invariant because $A_5$ is perfect, hence a multiple of $q$: $\mathsf S(Q)=Q$. The relations of $S_5=A_5\rtimes\langle t\rangle$ hold for $a\mapsto\rho(a)$, $t\mapsto\mathsf S$, so $S_5$ acts on $X$ by $k$-automorphisms; $A_5$ acts $k'$-linearly and odd elements $\sigma$-semilinearly. The action is faithful (a kernel meets $A_5$ trivially and is normal). Odd elements swap the two components of $X_{\bar k}$, so $X_{\free}$ is the complement of a finite set.

Step 2: the twist. Let $F_1^\flat=K(\sqrt{5\Delta})=K(\sqrt5\,\delta)$, quadratic over $K$. Since $k$ is algebraically closed in $L$, $\sqrt5\notin L$, so $N=L\otimes_KF_1^\flat=L(\sqrt5)$ is a field, $\Gal(N/K)=S_5\times\langle\sigma\rangle$, and $\Gal(N/F_1^\flat)=\{(g,\sigma^\epsilon):(-1)^\epsilon=\operatorname{sgn}g\}\cong S_5$, with $g(\sqrt5)=\operatorname{sgn}(g)\sqrt5$. This is the action of $G$ on $L\otimes_KF_1^\flat$ through the first factor, so $\tau\otimes_KF_1^\flat=\Spec N$ as $S_5$-torsors.
Now $X\times_kN=Q_{k'}\times_{k'}\Spec(k'\otimes_kN)$ and $k'\otimes_kN\cong N\times N$, $a\otimes n\mapsto(an,\sigma(a)n)$. An element $g\in S_5$ acts on $k'\otimes N$ by $g|_{k'}\otimes g|_N$ with $g|_{k'}=\sigma^{[g\text{ odd}]}$; the first coordinate of the image of $a\otimes n$ is $\sigma(a)g(n)=g(a)g(n)=g(an)$ for odd $g$ (and $g(an)$ for even $g$), so $G$ preserves both components. The first component is $Q_N$, and the second is its $\sigma$-conjugate, which is again $Q_N$ since $q$ has coefficients in $k$. By descent ${}^NX=Y_1\sqcup Y_2$ with $Y_i$ smooth projective geometrically integral of genus $0$ over $F_1^\flat$, i.e.\ conics. The conic $Y_1$ is embedded in $\PP^2_{F_1^\flat}$ by its anticanonical system, so a rational line cuts it in a subscheme of degree $2$, so there is a quadratic $F_2/F_1^\flat$ with $Y_1\cong\PP^1_{F_2}$; it has an $F_2$-point in the non-empty open $Y_1\cap{}^N(X_{\free})$. This is a $G$-equivariant $k$-morphism $\Spec(N\otimes_{F_1^\flat}F_2)\to X_{\free}$, and $N\otimes_{F_1^\flat}F_2=\tau\otimes_KF_2$. By Lemma~\ref{lem:S5free}, $\ed_k(\tau\otimes F_2)\le1$, and $[F_2:K]=4$.
\end{proof}

\begin{proof}[Proof of Theorem~\ref{thm:S5}]
If some $F'/K$ of degree $\le3$ had $\ed_k(\tau\otimes F')\le1$, Lemma~\ref{lem:S5bc} would give a field of degree $\le3$ over $\bar K$ with $\ed_{\bar k}\le1$, contradicting Proposition~\ref{prop:S5lower}. Proposition~\ref{prop:S5upper} gives $\edd4_k(\tgen)\le1$. Hence $d_1=4$. For $4\le d\le119$ the value is $\le1$ by monotonicity and $\ge1$ by Lemma~\ref{lem:S5last}, which also gives $\edd{120}_k(\tgen)=0$. Finally $\edd d_k(\tgen)\le\ed_k(S_5)=2$ for $d\le3$.
\end{proof}

\begin{corollary}[the profile of $A_5$]\label{cor:A5}
Let $k$ have characteristic $0$ and contain $\sqrt5$, and let $\tau_{A_5}$ be the generic $A_5$-torsor. Then $\edd1_k(\tau_{A_5})=\ed_k(A_5)=2$ \cite[Thm.~6.5]{BR97}, $\edd d_k(\tau_{A_5})=1$ for $2\le d\le59$ and $=0$ for $d\ge60$. In particular $d_1(\tau_{A_5})=2$ and $d_0=60$.
\end{corollary}
\begin{proof}
The generic $A_5$-torsor of $V|_{A_5}$ is $L/F_1$, and by the case $\sqrt5\in k$ of Proposition~\ref{prop:S5upper} there is a quadratic extension $F_2/F_1$ with $\ed_k(L\otimes_{F_1}F_2)\le1$; so $\edd2_k\le1$. The profile is non-increasing, $\edd1_k=\ed_k(A_5)=2$ by \cite{BR97}, $\edd d_k\ge1$ for $d<60$ and $\edd{60}_k=0$ by Proposition~\ref{prop:d0}(c).
\end{proof}

\begin{remark}\label{rem:strict}
The profile can drop strictly below every subgroup bound and every base-field bound. Take $k=\bar k$ of characteristic $0$ and $G=S_5$. Over an algebraically closed field the finite groups of essential dimension $1$ are the cyclic groups and the dihedral groups of order $2m$ with $m$ odd \cite[Thm.~6.2]{BR97}. The largest such subgroup of $S_5$ is $D_5$, of order $10$, so the bound of Corollary~\ref{cor:sub} for $j=1$ is $[S_5:D_5]=12$, and there are no non-trivial constant-field extensions of $k$. But $d_1(\tgen(S_5))=4$ (Theorem~\ref{thm:S5}): the optimal extension of degree $4$ is not a subgroup fixed field, since $\ed(A_5)=2$, and it is not a constant-field extension. It is a quadratic extension of the discriminant field $L^{A_5}$ that splits a twisted conic.
\end{remark}

\subsection{Symmetric groups of small degree}

\begin{proposition}\label{prop:Sn4}
Let $\operatorname{char}k=0$ and let $\tgen=\tgen(S_n)$.
\begin{enumerate}[label=\textup{(\alph*)}]
\item $n=2$: $\edd1_k=1$ and $\edd d_k=0$ for $d\ge2$.
\item $n=3$: $\edd d_k=1$ for $1\le d\le5$, $\edd d_k=0$ for $d\ge6$; $d_1=1$, $d_0=6$.
\item $n=4$: $\edd1_k=2$, $\edd d_k=1$ for $2\le d\le23$, $\edd d_k=0$ for $d\ge24$; $d_1=2$, $d_0=24$.
\end{enumerate}
\end{proposition}
\begin{proof}
(a) is $\Z/2$. (b) $\ed_k(S_3)=1$ (the generic cubic reduces to $x^3+ax+a$), so $\edd d_k\le1$, and the lower bound and $d_0=6$ are Lemma~\ref{lem:S5last}. (c) $\ed_k(S_4)=2$: the upper bound is \cite{BR97}, and the lower bound is Theorem~\ref{thm:rankindex} with $d=1$ and $A=V_4$. The rotation group of the cube is $S_4\subset\operatorname{SO}_3(\Q)$ (signed permutation matrices of determinant $1$), preserving $q=x_1^2+x_2^2+x_3^2$; so $S_4$ acts on the conic $Q=\{q=0\}\subset\PP^2$, faithfully and freely off a finite set (as in \S\ref{sec:S5}, The icosahedral conic). The twist ${}^\tau Q$ is a conic over $K$, which has a point over a quadratic extension $F'/K$ (a rational line meets it in a subscheme of degree $2$); then ${}^\tau Q\cong\PP^1_{F'}$ has an $F'$-point in the non-empty open set of free points, and Lemma~\ref{lem:S5free} gives $\ed_k(\tau\otimes F')\le1$. Hence $\edd2_k\le1$, and $\edd d_k=1$ for $2\le d\le23$ and $=0$ for $d\ge24$ by Lemma~\ref{lem:S5last}.
\end{proof}

\begin{example}[symmetric groups]\label{ex:Sn}
Let $\operatorname{char}k=0$, $L=k(x_1,\dots,x_n)$, $F=K=k(s_1,\dots,s_n)$, so $\tgen$ is the generic degree-$n$ \'etale algebra. For $F_m=F(x_1,\dots,x_m)$, $[F_m:F]=n(n-1)\cdots(n-m+1)$, Corollary~\ref{cor:sub} gives
\[
\edd{n(n-1)\cdots(n-m+1)}_k(\tgen)\le\ed_k(\tgen\otimes F_m)=\ed_k(S_{n-m}),
\]
the equality being Lemma~\ref{lem:ind} applied to the generic $S_{n-m}$-torsor.
With $\ed_{k,q}(S_n)\ge\lfloor n/q\rfloor$ (Lemma~\ref{lem:edqSn}),
Theorem~\ref{thm:div} gives $\edd d_k(\tgen)\ge\lfloor n/q_d\rfloor$ for the least prime $q_d>d$, so the profile of the generic \'etale algebra decays no faster than $\sim n/d$. For $S_6$, with $\ed_k(S_6)=3$, $\ed_k(S_5)=\ed_k(S_4)=2$, $\ed_k(S_3)=1$ (\cite[Thm.~6.5(d)]{BR97} for $S_4,S_5,S_6$; for $S_3$ the generic cubic reduces to $x^3+ax+a$):
$\edd1_k=3$; $\edd6_k\le2$, $\edd{120}_k\le1$, $d_0=720$. Since $\ed_{k,2}(S_6)\ge3$ and $\ed_{k,3}(S_6)\ge2$ (Lemma~\ref{lem:edqSn}), Theorem~\ref{thm:div} gives $2\mid d_2$ and $6\mid d_1$, hence $\edd d_k\ge2$ for $2\le d\le5$; so $d_2\in\{2,4,6\}$ and $d_1\in6\Z\cap[6,120]$. The discriminant field $L^{A_6}$, of degree $2$, is a witness for $d_2=2$ only if $\ed_k(A_6)\le2$ (Corollary~\ref{cor:sub} and Lemma~\ref{lem:ind}).
\end{example}

\begin{remark}\label{rem:S5coh}
The lower bound of Proposition~\ref{prop:S5lower} is obtained by Abel--Jacobi rigidity, not from cohomological invariants. We have not investigated whether cohomological invariants of positive degree of $\tgen(S_5)$ over quadratic extensions $K(\sqrt u)$ could detect it; the proof does not depend on the question.
\end{remark}

% ---- end S5.tex ----

% ---- begin openprob.tex ----
\section{Open problems}\label{sec:open}

We collect the questions that the results above leave open, with what is known in each case. Throughout, $\tgen$ is the generic torsor and $d_j=d_j(\tgen)$.

\subsection*{Cyclic groups of prime order in characteristic $0$}
Question~\ref{q:sharpodd} asks whether $d_{1,k}(\tgen(\Z/p))=m$ when $m=[k(\zeta_p):k]$ is odd and $\ge3$; Remark~\ref{rem:stalls} explains where the method stops. A second question concerns characteristic $\ell>0$, $\ell\ne p$: does Theorem~\ref{thm:Zp} hold there? The upper bound does (Corollary~\ref{cor:pertorsor}). The lower bound uses characteristic $0$ in four places: the regularity of the constant field (Step~1), the linearisation of the action at a fixed point (Step~5), tame Riemann--Hurwitz (Step~6) and the differential criterion for algebraic independence (Step~7).

\subsection*{The subgroup formula for $p$-groups}
\begin{question}\label{q:subgroupformula}
Let $G$ be a $p$-group, $\operatorname{char}k\ne p$, $\zeta_p\in k$. Does equality hold in \eqref{eq:subgroupformula} for all $d\ge1$, i.e.\ are the fixed fields $L^H$ of subgroups optimal witnesses?
\end{question}
Equality is known for $d<p$ and $d\ge|G|$, for every $d$ when $G$ is abelian of exponent $p^a$ and $\mu_{p^a}\subset k$ (Corollary~\ref{cor:subgroupabelian}), and for $\Z/4$ over $\Q$ (Remark~\ref{rem:Z4Q}); for $\Z/p^2$ it follows from Example~\ref{ex:rdim} and Theorem~\ref{thm:rankindex}. The first open case is $\Z/p^n$ with $n\ge3$: is $\ed_k(\tgen\otimes F')\ge p^{n-2}$ for every extension $F'/K$ of degree $p$? The smallest instance is $\Z/8$ over $\Q$ at $d=2$, where $1\le\edd2_\Q(\tgen)\le2$. A proof needs an incompressibility statement for equivariant correspondences of degree $p^j$ into free $G$-varieties of small dimension; Theorem~\ref{thm:rankindex} is vacuous for cyclic $G$ (Remark~\ref{rem:limits}), and the method of Theorem~\ref{thm:KM} controls only prime-to-$p$ degrees (Remark~\ref{rem:KMscope}). Non-cyclic non-abelian $p$-groups have not been examined.

\subsection*{Abelian groups without enough roots of unity}
\begin{question}\label{q:elemab}
Let $q$ be an odd prime with $\zeta_q+\zeta_q^{-1}\notin k$. Then $q^{n-j}\le d_j(\tgen((\Z/q)^n))\le[k(\zeta_q+\zeta_q^{-1}):k]\,q^{n-j}$ (Remark~\ref{rem:A8}(ii)). What is the true value?
\end{question}
The smallest case is $(\Z/5)^2$ over $\Q$, where $5\mid d_1$ by Theorem~\ref{thm:div}(b), so $d_{1,\Q}\in\{5,10\}$.

\subsection*{Symmetric and alternating groups}
\begin{question}\label{q:Sn}
(a) What is $d_1(\tgen(S_n))$ for $n\ge6$? We know $6\mid d_1$ (Corollary~\ref{cor:SnSylow}) and, for $n=6$, $d_1\le120$. Does $d_1$ grow with $n$, and is it attained by an extension of degree $O(n)$? (b) Determine $d_2(\tgen(S_6))\in\{2,4,6\}$ (Example~\ref{ex:Sn}). (c) What is $d_1(\tgen(A_5))$ over fields not containing $\sqrt5$? Over fields containing $\sqrt5$ it is $2$ (Corollary~\ref{cor:A5}); the semilinear construction of Proposition~\ref{prop:S5upper} uses a transposition and needs $S_5$.
\end{question}
For $S_5$ the Abel--Jacobi argument pins down $d_1$ because faithful $S_5$-curves of genus $\le1$ can be classified. For larger $n$ and degree $\ge4$ it needs information on faithful $S_n$-actions on trigonal and tetragonal curves, which we do not have.

\subsection*{Structure of the profile and comparison with branched covers}
\begin{question}\label{q:structure}
The upper bounds we use are subgroup fixed fields (Corollary~\ref{cor:sub}) and constant-field extensions (Proposition~\ref{prop:bfe}), which combine to $\edd{[k':k]\,[G:H]}_k(\tgen)\le\ed_{k'}(H)$. For which $(G,k)$ is the profile of $\tgen$ given by this combined bound? It is for abelian groups with enough roots of unity (Corollary~\ref{cor:subgroupabelian}) and is not for $S_5$ (Remark~\ref{rem:strict}). Is there a structural description of optimal extensions in general, for instance as a tower of subgroup fixed fields, constant-field extensions and splittings of twisted curves, as in Klein's reduction?
\end{question}
A related question is whether $\ed_k(S_5;\le4)=1$ for irreducible covers in the sense of \cite{FW25}; Theorem~C does not decide it, since Klein's witness is reducible.

% ---- end openprob.tex ----

% Bibliography.

\end{document}